\documentclass[a4paper,11pt]{amsart}
\usepackage{amssymb,amscd}
\usepackage{mathrsfs} 
\usepackage{tikz-cd}
\usepackage{enumitem}
\usepackage[capitalize]{cleveref}
\newcounter{alph}

\newtheorem{theo}[alph]{Theorem}
\newtheorem{propo}[alph]{Proposition}
\newtheorem{coro}[alph]{Corollary}
\numberwithin{equation}{section}
\newtheorem{cor}[equation]{Corollary}
\newtheorem{lem}[equation]{Lemma}

\newtheorem{prop}[equation]{Proposition}
\newtheorem{thm}[equation]{Theorem}

\theoremstyle{definition}

\newtheorem{exa}[equation]{Example}
\newtheorem{rem}[equation]{Remark}
\def\C{\mathbb C}
\def\F{\mathbb F}
\def\N{\mathbb N}
\def\R{\mathbb R}
\def\H{\mathbb H}

\def\ve{\varepsilon}
\def\vf{\varphi}
\def\la{\langle}
\def\ra{\rangle}

\newcommand{\ds}{\operatorname{ds}}
\newcommand{\dt}{\operatorname{dt}}
\newcommand{\dv}{\operatorname{\it dv}}

\newcommand{\ess}{\operatorname{ess}}

\newcommand{\im}{\operatorname{im}}

\newcommand{\Ray}{\operatorname{Ray}}

\newcommand{\SL}{\operatorname{SL}}
\newcommand{\SO}{\operatorname{SO}}
\newcommand{\Spin}{\operatorname{Spin}}
\newcommand{\spec}{\operatorname{spec}}
\newcommand{\supp}{\operatorname{supp}}
\newcommand{\tr}{\operatorname{tr}}

\begin{document}


\title[Essential spectrum of operators]
{On the essential spectrum of differential operators
over geometrically finite orbifolds}
\author{Werner Ballmann}
\address
{WB: Max Planck Institute for Mathematics,
Vivatsgasse 7, 53111 Bonn}
\email{hwbllmnn\@@mpim-bonn.mpg.de}
\author{Panagiotis Polymerakis}
\address{PP: Max Planck Institute for Mathematics,
Vivatsgasse 7, 53111 Bonn}
\email{polymerp\@@mpim-bonn.mpg.de}

\thanks{\emph{Acknowledgments.}
We are grateful to the Max Planck Institute for Mathematics
and the Hausdorff Center for Mathematics in Bonn for their support and hospitality.
We would like to thank Boris Hasselblatt for helpful information on references.}

\date{\today}


\begin{abstract}
We discuss the essential spectrum of essentially self-adjoint elliptic differential operators of first order
and of Laplace type operators on Riemannian vector bundles over geometrically finite orbifolds.
\end{abstract}
 
\maketitle

\tableofcontents

\section*{Introduction}
\label{secintro}

The essential spectrum of differential operators on Riemannian vector bundles over Riemannian
manifolds depends on the geometric structure of the bundles and manifolds at infinity.
We are interested in the case where the manifolds in question are \emph{geometrically finite}
in the sense of Bowditch \cite{Bowditch95}.
More generally, we will consider differential operators on vector bundles
over geometrically finite Riemannian orbifolds.
The reader not familiar with orbifolds may substitute \lq mani\rq{} for \lq orbi\rq{}
wherever the prefix \lq orbi\rq{} occurs. 

To set the stage,
let $O$ be a  complete and connected Riemannian orbifold of dimension $m$ with Levi-Civita connection $\nabla$,
curvature tensor $R$, and sectional curvature $-b^2\le K\le-a^2$, where $0<a\le b$.
Since $K\le0$, we have $O=\Gamma\backslash X$,
where $X$ is a complete and simply connected Riemannian manifold
and $\Gamma$ a countable group which acts properly discontinuously and isometrically on $X$.

Let $X_\iota$ be the \emph{ideal boundary} of $X$ and $X_c=X\cup X_\iota$
the compactification of $X$ with respect to the cone topology \cite{EberleinONeill73}.
Denote by $\Lambda$ and $\Omega=X_\iota\setminus\Lambda$ the \emph{limit set} and \emph{domain of discontinuity} of $\Gamma$,
respectively.
Recall that $\Lambda$ is a closed and $\Gamma$-invariant subset of $X_\iota$.

Let $E\to O$ be a Riemannian vector bundle over $O$ together with a metric connection,
also denoted by $\nabla$.
Let $A$ be an elliptic differential operator on $E$ which is \emph{essentially self-adjoint},
that is, with respect to the $L^2$-product, it is symmetric on $C^\infty_c(O,E)$
and has self-adjoint closure in $L^2(O,E)$.
For example, the \emph{Laplacian} $\Delta=\nabla^*\nabla$ is essentially self-adjoint.

Denote by $\spec(A,O)$ and  $\spec_{\ess}(A,O)$ the \emph{spectrum} and \emph{essential spectrum} of the closure of $A$,
and let $\lambda_0(A,O)$ and $\lambda_{\ess}(A,O)$ be the \emph{bottom}
of $\spec(A,O)$ and $\spec_{\ess}(A,O)$, respectively.

Lift $E$ and $A$ to $X$ and denote the lifts also by $E$ and $A$.
Assume that $A$ on $E$ over $X$ is also essentially self-adjoint
and use the analogous notation for the above spectral invariants.

We say that $E$ and $A$ are \emph{uniform} if, over $X$, they are invariant under the action of a group $G$
of automorphisms of $E$, which factors through a uniform action of isometries on $X$.
For example, the Hodge-Laplacians on the bundles of forms over quotients
of hyperbolic spaces are uniform.
The following assertion is probably known to experts (see for example \cite[Theorem C]{CarronPedon04})
and is the reason for the equality statements in our main results.

\begin{propo}\label{propo}
If $\Omega\ne\emptyset$ and $A$ is uniform, then
\begin{align*}
	\spec(A,X) = \spec_{\ess}(A,X) \subseteq \spec_{\ess}(A,O)
\end{align*}
and, in particular, $\lambda_{\ess}(A,O) \le \lambda_{\ess}(A,X) = \lambda_0(A,X)$.
\end{propo}

The chief issue in our main results, Theorems \ref{bottom} -- \ref{sigmal} below,
is the converse inclusion and inequality, respectively.

Recall that $\spec_{\ess}(A,O)=\emptyset$ in the case where $O$ is compact.
However, non-compact\-ness of $O$ is not a sufficient replacement
for the condition $\Omega\ne\emptyset$ in \cref{propo};
compare with \cref{subold} below.

\subsection{Main results}
\label{submain}
We say that $A$ is of \emph{Laplace type} if $A$ is of second order
and the principal symbol of $A$ satisfies $\sigma_A(\alpha)=-|\alpha|^2$.
This holds if and only if $A-\Delta$ is of (at most) first order.
\emph{Schr\"odinger operators} (with respect to $\nabla$) $A=\Delta+V$ are of Laplace type
since the \emph{potential} $V$ is of order zero.
Important examples are the squares of Dirac operators, like $(d+d^*)^2$ on differential forms,
where the potential is a curvature term.
If $A$ is of Laplace type and bounded from below, then $A$ is essentially self-adjoint;
see \cite{BallmannPolymerakis20a}.
(The manifold proof there extends readily to orbifolds.)

\begin{theo}\label{bottom}
Assume that $O$ is geometrically finite, that $A$ is of Laplace type,
and that $A$ is bounded from below on $E$ over $X$.
Then $A$ on $E$ over $O$ is also bounded from below and
\begin{align*}
	\lambda_{\ess}(A,O) \ge \lambda_{\ess}(A,X).
\end{align*}
Equality holds if the volume of $O$ is infinite and $A$ is uniform.
\end{theo}

In the case where $A$ is of first order,
we can say more about the essential spectrum,
at least when the principal symbol $\sigma_A$ is uniformly bounded.
Then both, $A$ on $E$ over $O$ and $A$ on $E$ over $X$, are essentially self-adjoint.

\begin{theo}\label{sigmaf}
Assume that $O$ is geometrically finite, that $A$ is of first order,
and that $\|\sigma_A\|_\infty<\infty$.
Then \[\spec_{\ess}(A,O)\subseteq\spec_{\ess}(A,X).\]
Equality holds if the volume of $O$ is infinite and $A$ is uniform.
\end{theo}

Dirac type operators, like $d+d^*$ on differential forms,
have a parallel principal symbol so that \cref{sigmaf} applies to them.

Controlling the essential spectrum of second order operators is more sophisticated.
Our arguments are based on standard estimates of Jacobi fields and non-standard
estimates of their variational derivatives.

Let $A=\Delta+B$ be a Laplace type operator on $E$.
Given a local orthonormal frame $(X_i)$ of $X$ or $O$,
write $B=\sum\sigma_B(X_i)\nabla_{X_i}+V$,
where the potential $V$ is a field of endomorphisms of $E$.
This decomposition of $B$ depends on the choice of the connection $\nabla$ on $E$;
however, $V$ does not depend on the choice of orthonormal frame.
If $V=V_++V_-$ denotes the decomposition in symmetric and skew-symmetric part,
then $\sum\sigma_B(X_i)\nabla_{X_i}+V_-$ is formally self-adjoint.
A straightforward computation gives $V_-=\frac12\sum(\nabla_{X_i}\sigma_B)(X_i)$.

\begin{theo}\label{sigmal}
Assume that $O$ is geometrically finite, that $\|\nabla R\|_\infty<\infty$,
and that $A=\Delta+B$ is a Laplace type operator with $\|\sigma_B\|_\infty<\infty$
and potential $V$ bounded from below.
Then both, $A$ on $E$ over $X$ and $E$ over $O$, are bounded from below
and \[\spec_{\ess}(A,O)\subseteq\spec_{\ess}(A,X).\]
Equality holds if the volume of $O$ is infinite and $A$ is uniform.
\end{theo}

A large class of operators, where \cref{sigmal} applies,
are  Schr\"odinger operators $A=\Delta+V$ with potential $V$ bounded from below.

\begin{coro}\label{sigmas}
Assume that $O$ is geometrically finite, that $\|\nabla R\|_\infty<\infty$,
and that $A=\Delta+V$ is a Schr\"o\-din\-ger operator with potential $V$ bounded from below.
Then \[\spec_{\ess}(A,O)\subseteq\spec_{\ess}(A,X).\]
Equality holds if the volume of $O$ is infinite and $A$ is uniform.
\end{coro}

\subsection{Hyperbolic orbifolds}
\label{subhype}
We say that an orbifold $O=\Gamma\backslash X$ is \emph{hyperbolic}
if $X$ is one of the hyperbolic spaces $X_\F^\ell$, where $m=d\ell$ with $d=\dim_\R\F$.
Since $\nabla R$ vanishes identically on hyperbolic spaces,
all the above results apply to hyperbolic orbifolds.

Let $X=X_\F^\ell$ and write $X=G/K$, where $(G,K)$ is a Riemannian symmetric pair,
$G$ a semi-simple Lie group, and $K$ the stabilizer in $G$ of a point $x_0\in X$.
Let $\pi$ be an orthogonal representation of $K$ on a finite-dimensional Euclidean space $E_0$
and $E_\pi$ be the \emph{homogeneous Riemannian vector bundle} over $X$ associated to $\pi$.
(See \cref{subcasi} for more details.)
Denote by $\nabla$ the metric connection on $E_\pi$ induced by the Levi-Civita connection of $X$.

Let $\Delta_\pi$ be the \emph{Casimir operator} on $E_\pi$.
In terms of an orthonormal basis $(Z_i)$ of $\mathfrak g$ with respect to the Killing form $B$ of $G$,
we have
\begin{align}\label{delpi}
	\Delta_\pi = - \sum B(Z_i,Z_i)\pi_*(Z_i)\pi_*(Z_i).
\end{align}
The Casimir operator is a Schr\"odinger operator with respect to $\nabla$.
Note that $\nabla$ and $\Delta_\pi$ are $G$-equivariant with respect to the left-action of $G$ on $E_\pi$.
In particular, the potential of $\Delta_\pi$ is bounded from below, hence also $\Delta_\pi$.

Let $\Gamma\subseteq G$ be a discrete subgroup and set $O=\Gamma\backslash X$,
a hyperbolic orbifold.
Let $\rho$ be an orthogonal representation of $\Gamma$ on $E$
such that $\pi(k)$ and $\rho(g)$ commute, for all $k\in K$ and $g\in\Gamma$.
Since $\Gamma$ acts on $E_\pi$ from the left via $\rho$,
preserving Riemannian metric, connection, and Casimir operator,
we can push down $E_\pi$ to the Riemannian vector bundle $E_{\pi,\rho}=\Gamma\backslash E_\pi$
over $O$ with metric connection $\nabla$ and \emph{twisted Casimir operator} $\Delta_{\pi,\rho}$.
Since $\Delta_{\pi,\rho}$ lifts to $\Delta_\pi$ and $\Delta_\pi$ is $G$-invariant,
$\Delta_{\pi,\rho}$ is a uniform differential operator as defined further up.

\begin{coro}\label{sigmac}
Let $O=\Gamma\backslash G/K$ be geometrically finite with infinite volume.
Then the twisted Casimir operator $\Delta_{\pi,\rho}$ is bounded from below and
\begin{align*}
	\spec_{\ess}(\Delta_{\pi,\rho},O) = \spec_{\ess}(\Delta_\pi,X) = \spec(\Delta_\pi,X).
\end{align*}
\end{coro}

\begin{exa}\label{hodge}
If $\pi$ is the isotropy representation of $K$ on the Euclidean space of alternating $k$-forms on $T_{x_0}X$,
where $0\le k\le m$, then $E_\pi=\Lambda^kX$ is the bundle of $k$-forms over $X$
and $\Delta_\pi=(d+d^*)^2$, the \emph{Hodge-Laplacian}, also denoted by $\Delta_k$. 
The spectrum of $\Delta_k$ on $X$ has been determined explicitly,
except--as far as we know--for the hyperbolic octonionic plane \cite{Donnelly81,Pedon98,Pedon99,Pedon05}.
In the case where $O=\Gamma\backslash G/K$ is geometrically finite with infinite volume,
we get from \cref{sigmac} (or also \cref{sigmas}) and loc.\,cit.\ that
\begin{align*}
	\spec_{\ess}(\Delta_{k,\rho},O) = \spec(\Delta_k,X)
	= \begin{cases}
	[\delta_k,\infty) &\text{if $k\ne m/2$},\\
	\{0\}\cup[\delta_k,\infty) &\text{if $k=m/2$,}
	\end{cases}
\end{align*}
where $\delta_k$ is the bottom of the continuous spectrum of $\Delta_k$ on $X=X_\F^\ell$.
If the Riemannian metric on $X$ is normalized so that its maximal sectional curvature is $-1$,
then we have the following explicit values for $\delta_k$:
\begin{enumerate}
\item\label{delta0}
$\delta_0=(m+d-2)^2/4$, the asymptotic volume growth of $X_\F^\ell$;
\item\label{deltakr}
for $\F=\R$, we have $\ell=m$,
$\delta_{k}=(k-(m-1)/2)^2$ for $0\le k\le m/2$,
and $\delta_{k}=(k-(m+1)/2)^2$ for $m/2\le k\le m$ \cite{Donnelly81,Pedon98};
\item\label{deltakc}
for $\F=\C$, we have $2\ell=m$,
$\delta_{k}=(k-\ell)^2$ for $0\le k\le m$, $k\ne\ell$,
and $\delta_{\ell}=1$ \cite{Pedon99}.
\end{enumerate}
For $\F=\H$, the formulas for the $\delta_{k}$ are a bit more involved \cite{Pedon05}.
For any of the four $\F$, $\delta_k=0$ occurs only in the case $\F=\R$ and $k=(m\pm1)/2$.
\end{exa}

\begin{exa}\label{dolbeault}
In the case of complex hyperbolic orbifolds $O=\Gamma\backslash G/K$,
if $\pi$ is the isotropy representation of $K$ on the Euclidean space of alternating $(p,q)$-forms on $T_{x_0}X$,
where $0\le p+q\le m$, then $E_\pi=\Lambda^kX$ is the bundle of $(p,q)$-forms over $X$
and $\Delta_\pi=2(\bar\partial+\bar\partial^*)^2$, the \emph{Dolbeault-Laplacian}, also denoted by $\Delta_{p,q}$.
If $O$ is geometrically finite with infinite volume
and the Riemannian metric on $X=X_\C^\ell$ is normalized so that its maximal sectional curvature is $-1$,
then
\begin{align*}
	\spec_{\ess}(\Delta_{p,q,\rho},O) = \spec(\Delta_{p,q},X)
	= \begin{cases}
	[(p+q-\ell)^2,\infty) &\text{if $p+q\ne\ell$},\\
	\{0\}\cup[1,\infty) &\text{if $p+q=\ell$,}
	\end{cases}
\end{align*}
by \cref{sigmac} (or also \cref{sigmas}) and \cite[Corollary 4.2]{Pedon99}.
\end{exa}

Suppose now that $A$ is a $G$-invariant elliptic differential operator on $E_\pi$ over $X=G/K$
of order one, which is symmetric on $C^\infty_c(X,E_\pi)$.
Then the principal symbol $\sigma_A$ of $A$ is a $G$-invariant, hence parallel
one-form on $X$ with values in the bundle of skew-symmetric endomorpohisms of $E_\pi$.
In particular, $A$ is essentially self-adjoint.
Writing $A$ in the form $A=\sum\sigma_A(X_i)\nabla_{X_i}+V$ as further up,
we get that the part $\sum\sigma_A(X_i)\nabla_{X_i}$ of $A$ of order one
is symmetric on $C^\infty_c(X,E_\pi)$ and conclude that the potential $V$ is a $G$-invariant,
hence parallel field of symmetric endomorphisms of $E_\pi$.
For $\Gamma$ and $\rho$ as above such that $\rho$ is compatible with $\sigma_A$ and $V$
in the sense of \cref{subcasi}, we get a uniform and essentially self-adjoint
differential operator $A_\rho$ on $E_{\pi,\rho}=\Gamma\backslash E_\pi$ over $O=\Gamma\backslash G/K$,
to which \cref{sigmaf} applies.

\begin{coro}\label{sigmad}
Let $O=\Gamma\backslash G/K$ be geometrically finite with infinite volume.
Then $A_{\rho}$ is essentially self-adjoint with
\begin{align*}
	\spec_{\ess}(A_{\rho},O) = \spec_{\ess}(A,X) = \spec(A,X).
\end{align*}
\end{coro}

\begin{exa}\label{baer}
Write $X=H_\R^2=G/K$, where $G=\SL(2,\R)$
and $K=\SO(2)$ is the stabilizer of a point $x_0\in X$.
Then $G$ can be identified with the unique spin structure over $X$,
and the isotropy action of $K$ on $T_{x_0}X$ is a twofold cover of the action of $\SO(T_{x_0}X)$
such that $K\cong\Spin(2)$.
The Dirac operator $D$ on the spinor bundle $E_\pi$,
the bundle associated to the spin-representation $\pi$ of $K$, is $G$-invariant
and of the form $D=\sum X_i\nabla_{X_i}$,
where we recall that the symbol $\sigma_D$ of $D$ is given by Clifford multiplication,
$\sigma_D(X)u=Xu$.
We have
\begin{align*}
	\spec_{\ess}(D,X) = \spec(D,X) = \R;
\end{align*}
see \cite[p.\,441]{Baer99} for references and discussion.
Hence $\spec_{\ess}(D_{\rho},O)=\R$ as long as $O=\Gamma\backslash X$
is geometrically finite with infinite area;
compare with the corresponding remarks in \cref{subold}.
\end{exa}

\subsection{Earlier work and comments}
\label{subold}
In \cite{McKean70}, McKean showed that the spectrum of the Laplacian $\Delta_0$ on $X$
is bounded from below by $(m-1)^2/4a^2$ (not assuming a lower bound on the curvature).
\cref{bottom} implies that the same bound holds for the essential spectrum
of $\Delta_0$ on geometrically finite orbifolds.
For geometrically finite manifolds,
this estimate is contained in Hamenst\"adt's \cite[Theorem]{Hamenstaedt04}.

In the case of hyperbolic orbifolds $\Gamma\backslash X_\F^\ell$
with Riemannian metric on $X_\F^\ell$ normalized so that its maximal sectional curvature is $-1$,
\cref{sigmac} (or \cref{sigmas}) implies that $\delta_0=(m+d-2)^2/4$ is a lower bound
for the essential spectrum of $\Delta_0$ on geometrically finite orbifolds.
For geometrically finite manifolds,
this estimate is contained in Hamenst\"adt's \cite[Corollary]{Hamenstaedt04}.

The case of $\Delta_k$ on geometrically finite real hyperbolic manifolds in \cref{hodge}
is \cite[Theorem 1.11]{MazzeoPhillips90} of Mazzeo and Phillips,
except for the additional assertion in loc.\,cit.\
that the essential spectrum of $\Delta_{m/2}$ is equal to $[1/4,\infty)$
in the case where $m$ is even and the manifold is non-compact, but of finite volume.
The case of $\Delta_0$ on geometrically finite hyperbolic manifolds in \cref{hodge}
is \cite[Remark 1.2]{Li20} of Li.

In \cite{BunkeOlbrich00},
Bunke and Olbrich obtain the Plancherel decomposition of $L^2(G,\rho)$
for convex cocompact hyperbolic orbifolds
(with an extra condition on the critical exponent of $\Gamma$ in the case of the octonionic plane),
where $\rho$ is a finite-dimensional orthogonal representation of $\Gamma$.
Recall here that convex cocompact means that the convex core is compact
and that convex cocompact orbifolds are geometrically finite.

In \cite[Theorems B and C]{CarronPedon04},
Carron and Pedon obtain lower bounds for the spectrum of $\Delta_k$
on hyperbolic manifolds respectively determine it completely, not assuming geometric finiteness,
but an upper bound for the critical exponent of $\Gamma$.
Note however that manifolds with negatively pinched sectional curvature
and sufficiently small critical exponent are convex cocompact \cite[Theorem 1.2]{LiuWang20}.

Concerning the assumption of infinite volume in our results,
consider the Dirac operator on $X=H_\R^2$ as in \cref{baer}.
If $O=\Gamma\backslash X$ is non-compact, but of finite area,
then the essential spectrum of $D_{\rho}$ depends on $\rho$.
Namely, $\spec_{\ess}(D_{\rho},O)=\emptyset$ if the spin structure is non-trivial
along the cusps of $O$ and $\spec_{\ess}(D_{\rho},O)=\R$ otherwise;
see B\"ar's \cite[Theorem 1]{Baer99},
which also covers the case of real hyperbolic spin manifolds of higher dimension and finite volume.
See also Lott's \cite[Theorem 2]{Lott01} and \cite[Theorem 5]{Lott02} with corresponding,
but less explicit results on Hodge-Laplacians and Dirac type operators.

Our approach was motivated by the earlier version \cite{BallmannPolymerakis20b}
of the present text and by \cite{Li20}.
In both sources, finite open coverings and subordinate partititons of unity
adapted to the geometry of $O$ at infinity play an important role.
As in \cite[Section 5.3]{Li20},
we use square roots of the functions belonging to these partitions of unity as cutoff functions.

\subsection{Structure of the article}
\label{substr}
In our notation concerning geometrically finite orbifolds, we mostly follow \cite{Bowditch95}.
After preparations in \cref{secpat},
we review, in \cref{secgefi}, some important features of geometrically finite orbifolds.
In \cref{secpu}, we obtain specific finite open coverings of $O$
and associated partitions of unity,
which are at the heart of the proof of the main results in \cref{secray}.
\cref{secjac2} is devoted to a somewhat sophisticated comparison result
about variational derivatives of Jacobi fields.
It is here that the assumption on $\nabla R$ in \cref{sigmal} comes in.
In \cref{secmoll},
we extend a smoothing result for convex sets of Parkkonen-Paulin \cite{ParkkonenPaulin12},
where our emphasis is on uniformity in the case where the convex set may be non-compact. 
 In \cref{symmetry},
 we discuss the symmetry of second derivatives of $C^{1,1}$-functions almost everywhere,
 an issue, for which we could not identify a suitable reference. 

\section{Preliminaries and terminology}
\label{secpat}

Throughout the article, we let $X$ be a complete and simply connected Riemannian manifold
with sectional curvature $-b^2\le K=K_X\le-a^2<0$.
We denote by $X_\iota$ the \emph{ideal boundary} of $X$ and $X_c=X\cup X_\iota$
the compactification of $X$ with respect to the cone topology \cite{EberleinONeill73}.

For a discrete group $\Gamma$ of isometries of $X$,
we let $\Lambda=\Lambda_\Gamma$ be the \emph{limit set of $\Gamma$},
a closed and $\Gamma$-invariant subset of $X_\iota$.
Then $\Omega=\Omega_\Gamma=X_\iota\setminus\Lambda$
is called the \emph{domain of discontinuity of $\Gamma$}.
The action of $\Gamma$ on $X\cup\Omega$ is properly discontinuous
and $M_c(\Gamma)=\Gamma\backslash(X\cup\Omega)$ is a topological orbifold.

\subsection{Elementary groups of isometries}
\label{susegi}
A discrete group $\Gamma$ of isometries of $X$ is said to be \emph{elementary}
if it is of one of the following three types:
\begin{enumerate}
\item\label{ell}
\emph{elliptic}: $\Gamma$ fixes a point in $X$;
\item\label{lox}
\emph{loxodromic}: $\Gamma$ fixes a geodesic of $X$ as a set and no point in $X$;
\item\label{par}
\emph{parabolic}: $\Gamma$ fixes a point $p\in X_\iota$ and horospheres about $p$ as sets,
but fixes no other point of $X_c$.
\end{enumerate}
With respect to \eqref{ell},
recall that a discrete group $\Gamma$ of isometries of $X$ fixes a point of $X$
if and only if it is finite.
Therefore loxodromic and parabolic groups are infinite.
By \cite[Propositions 4.1 and 4.2]{Bowditch95},
parabolic groups are finitely generated, virtually nilpotent, and contain parabolic elements.

The following result is discussed in \cite[§7E]{BallmannGromovSchroeder85} in the case,
where $\Gamma$ acts freely on $X$.
As stated, it is Proposition 3.1.1 in \cite{Bowditch95}.

\begin{prop}\label{lemele}
Discrete virtually nilpotent groups of isometries of $X$ are elementary.
\end{prop}

Recall also that $\Gamma$ is elementary if and only if $|\Lambda_\Gamma|<\infty$,
where $|\Lambda_\Gamma|=0$ if $\Gamma$ is elliptic, $|\Lambda_\Gamma|=1$ if $\Gamma$ is parabolic,
and $|\Lambda_\Gamma|=2$ if $\Gamma$ is loxodromic.

\subsection{Distance to convex subsets}
\label{subconv}
Let $C$ be a closed and convex subset of $X$.
For the convenience of the reader,
we collect some results about the distance function to $C$.

For each $x\in X$, there is a unique point $\pi_C(x)\in C$ such that
\begin{align}\label{propi}
	d(x,C) = d(x,\pi_C(x)).
\end{align}
We call $\pi_C\colon X\to C$ the \emph{(orthogonal) projection to $C$}
and let $f\colon O\to\R$ be the distance function to $C$, $f(x)=d(x,C)$.
Then $f$ is convex and admits Lipschitz constant one.
Furthermore, the sublevels $C_r=\{f\le r\}$ of $f$ are convex for all $r>0$.

For each $x\in X\setminus C$, there is a unique unit speed geodesic $c_x\colon[0,\infty)\to X$
from $c_x(0)=\pi_C(x)\in C$ through $x=c_x(r)$, where $r=d(x,C)$.
By definition $f(c_x(t))=t$ for all $t\ge0$.
Since $f$ admits Lipschitz constant one, the first variation formula implies therefore that
\begin{align*}
	f(c(s)) - f(x)
	= \la c_x'(r),c'(0)\ra s + o(s),
\end{align*}
for any smooth curve $c$ through $x$ and sufficiently small $s$.
By uniqueness, $c_x'(r)$ depends continuously on $x$,
and hence $f$ is $C^1$ on $X\setminus C$ with gradient $\nabla f|_x=c_x'(r)$.
Now a classical argument from convex geometry extends to the Riemannian setting
and shows that $f$ is $C^{1,1}$ on $X\setminus C$.
More precisely, $f$ is twice differentiable exactly at the points of $X\setminus C$
at which $\pi_C$ is differentiable.

\begin{proof}
The map $\Phi\colon TX\to X\times X$, $\Phi(v)=(p(v),\exp(v))$, is a diffeomorphism,
where $p$ denotes the projection to the foot point.
Since
\begin{align}\label{lcpi}
	\nabla f(x) = \frac{-1}{|\Phi^{-1}(x,\pi_C(x))|} \Phi^{-1}(x,\pi_C(x))
\end{align}
for any $x\in X\setminus C$ and $\pi_C$ is Lipschitz continuous,
we conclude that $\nabla f$ is $C^{0,1}$ on $X\setminus C$.
Moreover, \eqref{lcpi} also implies that $f$ is twice differentiable
exactly at the points of $X\setminus C$ at which $\pi_C$ is differentiable.
\end{proof}

We say that $x\in X\setminus C$ is \emph{regular} if $f$ is twice differentiable at $x$
and the second derivative $\nabla^2f|_x$ is symmetric.
From \cref{symmetry}, we obtain that almost every point of $X\setminus C$ is regular.

\begin{lem}\label{lemtec}
Let $V=V(s)$ be a curve of tangent vectors on $X$ which is differentiable at $s=0$.
For all $s$, let $\gamma_s$ be the geodesic with initial velocity $V(s)$.
Then $J(t)=\partial\gamma_s(t)/\partial s|_{s=0}$ exists for all $t\in\R$,
and $J$ is the Jacobi field along $\gamma_0$ such that
\begin{align*}
  J'
  = \left.\frac{\nabla}{\partial t}\frac{\partial\gamma}{\partial s}\right|_{s=0}
  = \left.\frac{\nabla}{\partial s}\frac{\partial\gamma}{\partial t}\right|_{s=0}.
\end{align*}
\end{lem}

The point of this lemma is that, in the usual setup, the curve $V$ is assumed to be smooth.
Then $\nabla\partial\gamma/\partial s\partial t=\nabla\partial\gamma/\partial t\partial s$,
and the assertion of \cref{lemtec} follows easily.
Here we assume less, and a little extra thought is needed.

\begin{proof}
Let $p\colon TX\to X$ be the projection to the foot point
and $(F_t)$ be the geodesic flow of $X$.
Then $\gamma_s(t)=p(F_t(V(s)))$ and hence
\begin{align*}
	\left.\frac{\partial\gamma_s(t)}{\partial s}\right|_{s=0}
	= p_*F_{t*}(V'(0)).
\end{align*}
Hence we may replace $V$ by a smooth curve with the same derivative at $s=0$
to get that $J(t)$ exists for all $t\in\R$ and that it is equal to the asserted Jacobi field.
\end{proof}

Let $x=c_x(r)\in X\setminus C$ be a regular point.
For $u\in T_xX$, let $J_u$ be the Jacobi field along $c=c_x$ with
\begin{align*}
  J_u(r) = u
  \hspace{2mm}\text{and}\hspace{2mm}
  J_u'(r) = \nabla_u\nabla f.
\end{align*}

\begin{cor}\label{lemtec2}
For all $t>0$, $c_x(t)$ is a regular point; in fact,
\begin{align*}
  \nabla^2f(J_u(t),J_v(t)) = \la J_u(t),J_v'(t)\ra.
\end{align*}
Furthermore, $\pi_{C*}(J_u(t))=J_u(0)$.
\end{cor}

We also write $J(t)u=J_u(t)$.
Then $J(t)\colon T_xX\to T_{c_x(t)}X$ is an isomorphism, for all $t>0$.
Furthermore, the covariant derivative of $\nabla f$ satisfies
\begin{align}
	S(t) := \nabla^2 f|_{c_x(t)} = J'(t)J(t)^{-1},
\end{align}
by \cref{lemtec2}.
Note that $S$ is a symmetric field of endomorphisms along $c=c_x$
that satisfies the Riccati equation
\begin{align}\label{riccati}
	S' + S^2 + R_c = 0,
\end{align}
where $R_cu=R(u,c')c'$.
Clearly, $c'=\nabla f$ belongs to the kernel of $S$.
Therefore we discuss $S$ only on the normal bundle of $c$,
identifying the various $c'(t)^\perp$ with $c'(0)^\perp$ via parallel translation along $c$.
By \cite[p.\,212]{EschenburgHeintze90}, 
$S$ has the asymptotic behaviour
\begin{align}\label{PQ}
	S(t) = \frac{1}{t} P+ Q(t) \hspace{2mm}\text{as $t\to0$},
\end{align}
where $P$ is an orthogonal projection on $c'(0)^\perp$
and $Q$ extends continuously to $t=0$, such that $\im P\subseteq\ker Q(0)$.
We call the pair $(P,Q(0))$ the \emph{initial condition} of $S$ since it determines $S$ uniquely.
In terms of $S$, the space of Jacobi fields along $c$ which we consider
is given by the initial conditions
\begin{align}\label{inicon}
	J_v(0) = (1-P)v,\;
	J_v'(0) = Pv + Qv,
\end{align}
where $v\in c'(0)^\perp$.
By the convexity of $C$, we have $Q(0)\ge0$.

For any solution $S$ of \eqref{riccati} on the normal bundle of $c$,
let $S_a$ and $S_b$ be the fields of endomorphisms  on the normal bundle of $c$
which correspond to solutions of \eqref{riccati}
for constant sectional curvature $-a^2$ and $-b^2$, respectively,
with the same initial condition, $(P,Q(0))$.
More precisely,
let $(v_i)$ be an orthonormal basis of $\dot c(0)^\perp$ such that $v_2,\dots,v_k$ span $\ker P$
and are eigenvectors of $Q(0)$ with corresponding eigenvalue $\alpha_2,\dots,\alpha_k$
and $v_{k+1},\dots,v_m$ span $\im P$.
Then
\begin{align*}
	S_a(t)V_i(t) =
	\begin{cases}
	a\frac{\sinh(at)+\alpha_i\cosh(at)/a}{\cosh(at)+\alpha_i\sinh(at)/a}V_i(t)
	&\text{for $i\le k$},\\
	a\frac{\cosh(at)}{\sinh(at)}V_i(t)
	&\text{for $i> k$},	
	\end{cases}
\end{align*}
where the $V_i$ are parallel along $c$ with $V_i(0)=v_i$,
and similarly for $S_b$, substituting $b$ for $a$.
If $Q(0)\ge0$, as in the case $S=\nabla^2f|_{\dot c^\perp}$ under discussion,
we have $\alpha_i\ge0$, and then $S_a$ and $S_b$ are defined for all $t>0$.
Now \cite[Theorem, page 210]{EschenburgHeintze90} yields the following estimates.

\begin{lem}\label{curvest}
If $Q(0)\ge0$, then we have, for all $t>0$,
\begin{align*}
	S_a(t) \le S(t)|_{\nabla f(x)^\perp} \le S_b(t). 
\end{align*}
In particular, for any Jacobi field $J=J_u$ as above and perpendicular to $c$,
\begin{align*}
	a\tanh(at) \le (\ln|J|)'(t) \le b\coth(bt).
\end{align*}
\end{lem}

\begin{cor}\label{curvest2}
If $(P,Q(0))=(0,Q(0))$ with $0\le\alpha\le Q(0)\le\beta$, then
\begin{align*}
	a\frac{\sinh(at)+\alpha\cosh(at)/a}{\cosh(at)+\alpha\sinh(at)/a}
	\le S(t)
	\le b\frac{\sinh(bt)+\beta\cosh(bt)/b}{\cosh(bt)+\beta\sinh(bt)/b}.
\end{align*}
\end{cor}

The initial condition $(P,Q(0))=(0,0)$ yields the following estimates.

\begin{cor}[Rauch II]\label{rauch2}
If $J$ is a Jacobi field along $c$ with $J'(0)=0$, then we have, for all $t>0$,
\begin{align*}
	|J'(t)| \le b\tanh(bt)|J(t)|
	\quad\text{and}\quad
	|J(t)| \le \cosh(bt) |J(0)|.
\end{align*}
\end{cor}

\begin{rem}\label{equiv}
If $\Gamma$ is a group of isometries of $X$ which leaves $C$ invariant,
then all the above constructions and assertions are $\Gamma$-equivariant
and have their analogues in $\Gamma\backslash X$.
\end{rem}

\subsection{Amenable coverings}
\label{subame}
In the proof of our main results we will use \cite[Propositions 4.12 and 4.13]{Polymerakis20a},
which we summarize as follows:

\begin{prop}\label{pp}
Let $p\colon M_2\to M_1$ be an infinite and amenable Riemannian covering of Riemannian manifolds.
Let $A_1$ be a formally self-adjoint differential operator on a vector bundle $E_1$ over $M_1$
and $A_2$ be the lift of $D_1$ to the lift $E_2$ of $E_1$.
Then, for any $u_1\in C^\infty_c(M_1,E_1)$, $\lambda\in\R$, and $\ve>0$,
there exists a $u_2\in C^\infty_c(M_2,E_2)$ with $\|u_2\|_2 = \|u_1\|_2$ such that
\begin{align*}
	\supp u_2 &\subseteq p^{-1}(\supp u_1), \\
	\|(A_2-\lambda)u_2\|_2 &\le \|(A_1-\lambda)u_1\|_2+\ve \\
	\la A_2u_2,u_2\ra_2 &\le \la A_1u_1,u_1\ra_2+\ve.
\end{align*}
\end{prop}

The proof consists of a sophisticated choice of cutoff functions
to turn the lift of $u_1$ to $M_2$ into a section with the asserted properties.
Amenability makes such choices possible.

\subsection{Homogeneous vector bundles}
\label{subcasi}
Let $X$ be a symmetric space of non-compact type and write $X=G/K$,
where $(G,K)$ is a Riemannian symmetric pair,
$G$ a semi-simple Lie group, and $K$ the stabilizer in $G$ of a point $x_0\in X$.
Denote by $B$ the Killing form of $G$, and identify $T_{x_0}X$ as usual
with the $B$-orthogonal complement $\mathfrak p$
of the Lie algebra $\mathfrak k$ of $K$ in the Lie algebra $\mathfrak g$ of $G$.
The restriction of $B$ to $\mathfrak p$
induces a $G$-invariant Riemannian metric on $X$.

Let $\pi$ be an orthogonal representation of $K$ on a finite-dimensional Euclidean space $E_0$.
Denote by $E_\pi$ the Riemannian vector bundle over $X$ associated to $\pi$,
\begin{align*}
	E_\pi = \{ [g,u] \mid g\in G, u\in E_0 \},
\end{align*}
where $[gk,u]=[g,\pi(k)u]$ for all $g\in G$, $k\in K$, and $u\in E_0$.
Sections of $E_\pi$ are in one-to-one correspondence with maps $u\colon G\to E_0$
such that $u(gk)=\pi(k)u(g)$ for all $g\in G$ and $k\in K$.
Clearly, $g[h,u]=[gh,u]$ is a left-action of $G$ on $E_\pi$.
We call $E_\pi$ the \emph{homogeneous vector bundle} over $X$ associated to $\pi$.
Since $\pi$ is orthogonal, the inner product on $E_0$ induces a $G$-invariant Riemannian metric on $E_\pi$.
Conversely, if $E$ is a Riemannian vector bundle over $X$ with an associated orthogonal action of $G$
and $E_0$ is the fiber of $E$ over $x_0$,
then the isotropy representation $\pi$ of $K$ on $E_0$ yields an isometric isomorphism $E\cong E_\pi$.

The Levi-Civita connection of $X$ induces a $G$-invariant metric connection $\nabla$ on $E_\pi$.
The covariant derivative of a section $[g\exp(tX),u(t)]$ along the geodesic $g\exp(tX)x_0$
through $gx_0$, where $t\in\R$ and $X\in\mathfrak p$, is given by 
\begin{align}\label{nabla}
	[ge^{tX},u(t)]'(0) = [g,u'(0)].
\end{align}
We see that the section is parallel along the geodesic if and only if $u$ is constant.

\begin{lem}\label{nabla2}
Let $Z\in\mathfrak g$ and write $Z=X+Y$ with $X\in\mathfrak k$ and $Y\in\mathfrak p$.
Then the covariant derivative of a section $[g\exp(tZ),u(t)]$ along the curve $g\exp(tZ)x_0$
at $t=0$ is given by 
\begin{align*}
	[ge^{tZ},u(t)]'(0) = [g,u'(0)+\pi_*X(u(0))].
\end{align*}
\end{lem}

\begin{proof}
The curves $g\exp(tZ)x_0$ and $g\exp(tY)\exp(tX)x_0$ through $gx_0$
have the same derivative at $t=0$ and therefore
\begin{align*}
	[ge^{tZ},u(t)]'(0)
	&= [ge^{tY}e^{tX},u(t)]'(0) \\
	&= [ge^{tY},\pi(e^{tX})u(t)]'(0)
	= [g,u'(0) + \pi_*X(u(0))],
\end{align*}
where we use \eqref{nabla} in the last step.
\end{proof}

The \emph{Casimir operator} on $E_\pi$ is given by
\begin{align}\label{deltapi}
	\Delta_\pi = - \sum b^{ij}\pi_*(Z_i)\pi_*(Z_j),
\end{align}
where $(Z_i)$ is a basis of $\mathfrak g$ with respect to $B$
and $(b^{ij})$ is the inverse matrix of the matrix with entries $b_{ij}=B(Z_i,Z_j)$.
The Casimir operator acts on sections $u$ of $E_\pi$ by
\begin{align}\label{deltapi2}
	\Delta_\pi u(gx_0)
	= \left. - \sum b^{ij} \frac{\partial^2}{\partial s\partial t}u(ge^{sZ_i}e^{tZ_j}x_0)\right|_{s=t=0}.
\end{align}
The Casimir operator is a $G$-invariant Schr\"odinger operator with respect to the given connection of $E_\pi$,
where $\nabla^*\nabla$ and potential $V$ correspond to the parts
\begin{align}\label{deltapi4}
	\nabla^*\nabla = - \sum \pi_*(X_i)\pi_*(X_i)
	\hspace{3mm}\text{and}\hspace{3mm}
	V = \sum \pi_*(Y_j)\pi_*(Y_j)
\end{align}
of $\Delta_\pi$.
Here $(X_i)$ and $(Y_j)$ are $B$-orthonormal bases of $\mathfrak p$ and $\mathfrak k$.

Let $\Gamma\subseteq G$ be a discrete subgroup
and $O=\Gamma\backslash X$ the associated orbifold.
Let $\rho$ be an orthogonal representation of $\Gamma$ on $E_0$ such that
\begin{align}\label{rhopi}
	\pi(k)\rho(g) = \rho(g)\pi(k)
\end{align}
for all $k\in K$ and $g\in\Gamma$.
The standard example here is that $\rho$ is an orthogonal representation of $\Gamma$
on a Euclidean space $F_0$, where $E_0$ is replaced by $E_0\otimes F_0$
and $\pi$ and $\rho$ are extended to $E_0\otimes F_0$
by $\pi(g)(u\otimes v)=\pi(g)u\otimes v$ and $\rho(g)(u\otimes v)=u\otimes\rho(g)v$
for all $g\in G$ and $g\in\Gamma$, respectively. 
Since $\Gamma$ acts on $E_\pi$ from the left via
\begin{align}\label{rhopi2}
	g[h,u] = [gh,\rho(g)u],
\end{align}
preserving Riemannian metric, connection, and Casimir operator,
we can push down $E_\pi$ to the Riemannian vector bundle $E_{\pi,\rho}=\Gamma\backslash E_\pi$
over $O$ with metric connection $\nabla$ and \emph{twisted Casimir operator} $\Delta_{\pi,\rho}$.
Since $\Delta_{\pi,\rho}$ lifts to $\Delta_\pi$ and $\Delta_\pi$ is $G$-invariant,
$\Delta_{\pi,\rho}$ is a uniform differential operator as defined in the introduction.

In examples, the following two computations,
extracted from the proof of \cite[Lemma 5.2]{MoroianuSemmelmann10}, are useful.
Recall that, under the natural identification of $T_{x_0}X$ with $\mathfrak p$,
the isotropy representation of $\mathfrak k$ is given by the adjoint representation on $\mathfrak p$.
Let $(X_i)$ be an orthonormal basis of $\mathfrak p$ with respect to $B$.
Then 
\begin{align}\label{kp}
\begin{split}
	2[Y,X]
	&= \sum \{ B(X_i,X)[Y,X_i] - B([Y,X_i],X)X_i \} \\
	&= \sum (X_i\wedge[Y,X_i])(X),
\end{split}
\end{align}
for any $Y\in\mathfrak k$ and $X\in\mathfrak p$. 
Assume now that $\pi$ is equal to the composition of the adjoint representation of $K$ on $\mathfrak p$
with an orthogonal representation $\alpha$ of $\SO(\mathfrak p)$.
Let $(Y_k)$ be an orthonormal basis of $\mathfrak k$ with respect to $B$.
Then
\begin{align}\label{alpha}
\begin{split}
	V(x_0)
	&= \sum_k \pi_*(Y_k)\pi_*(Y_k) \\
	&= \frac12 \sum_{i,k} \alpha_*(X_i\wedge[Y_k,X_i])\pi_*(Y_k) \\
	&= \frac12 \sum_{i,j,k} B([Y_k,X_i],X_j)\alpha_*(X_i\wedge X_j)\pi_*(Y_k) \\
	&= \frac12 \sum_{i,j,k} B(Y_k,[X_i,X_j])\alpha_*(X_i\wedge X_j)\pi_*(Y_k) \\
	&= - \frac12 \sum_{i,j} \alpha_*(X_i\wedge X_j)\pi_*([X_i,X_j]) \\
	&= \frac12 \sum_{i,j} \alpha_*(X_i\wedge X_j)\alpha_*(R(X_i,X_j)).
\end{split}
\end{align}
We see that the potential is a curvature term.
An example is the representation $\alpha$ of $\SO(\mathfrak p)$
on the space of alternating $k$-forms on $\mathfrak p$.
Then the Casimir operator is equal to the Hodge-Laplacian, $\Delta_\pi=(d+d^*)^2=\Delta_k$.
As a consequence in the case of complex hyperbolic spaces,
if $\pi$ is the representation of $K$ on the space of forms on $\mathfrak p\otimes\C$ of type $(p,q)$,
then the Casimir operator is equal to the Dolbeault-Laplacian,
$\Delta_\pi=2(\bar\partial+\bar\partial^*)^2=\Delta_{p,q}$.

Let now $\pi$ be as above and $\sigma_0$ be a linear map on $\mathfrak p$
with values in the space of skew-symmetric endomorphisms of $E_0$.
Extend $\sigma_0$ to a one-form $\sigma$ on $X$
with values in the space of skew-symmetric endomorphisms of $E_\pi$ by
\begin{align}\label{asigm}
	\sigma_g(g_*X)[g,u] = [g,\sigma_0(X)u].
\end{align}
For $\sigma$ to be well-defined, we need that
\begin{align}\label{asigm2}
	\sigma_0(X)\pi(k) = \pi(k)\sigma_0(X)
\end{align}
for all $k\in K$ and $X\in\mathfrak p$.
Then
\begin{align}\label{asigm4}
	A_\sigma = \sum \sigma(X_i)\nabla_{X_i},
\end{align}
where $(X_i)$ is a local orthonormal frame of $X$,
is a $G$-invariant differential operator on $E_\pi$ of order one with principal symbol $\sigma$,
which is symmetric on $C^\infty_c(X,E_\pi)$.
Conversely, up to a $G$-invariant symmetric potential,
any $G$-invariant differential operator on $E_\pi$ of order one,
which is symmetric on $C^\infty_c(X,E_\pi)$, is of this type.

Given $\Gamma$ and $\rho$ as above, $\rho$ induces a twisted version $A_{\sigma,\rho}$
of $A_\sigma$ on $E_{\pi,\rho}$ over $O$ if and only if
\begin{align}
		\sigma_0(X)\rho(g) = \rho(g)\sigma_0(X)
\end{align}
for all $g\in\Gamma$ and $X\in\mathfrak p$.
Note that $A_\sigma$ and $A_{\sigma,\rho}$ are elliptic, and then also essentially self-adjoint,
if $\sigma_0$ satisfies the ellipticity condition that $\sigma_0(X)$ is invertible for $X\ne0$.

Examples of elliptic differential operators $A_\sigma$ as above are
the Hodge-Dirac operator $d+d^*$,
the Dolbeault-Dirac operator $\sqrt2(\bar\partial+\bar\partial^*)$,
and, if the isotropy representation of $K$ on $\mathfrak p=T_{x_0}X$ lifts to $\Spin(\mathfrak p)$,
the Dirac operator on the spinor bundle.

\section{Geometrically finite orbifolds}
\label{secgefi}

Recall that $O=\Gamma\backslash X$ is called \emph{convex cocompact}
if $M_c(\Gamma)=\Gamma\backslash(X\cup\Omega)$ is compact.
More generally and following Bowditch \cite[Definition on p.\,265]{Bowditch95},
we say that $O$ is \emph{geometrically finite} if $M_c(\Gamma)$ has at most finitely many ends
and each end of $M_c(\Gamma)$ is parabolic;
the latter notion is reviewed in \cref{susepe} below.

\subsection{Convex core}
\label{susecc}
For any two points $x,y\in X_c$, we denote by $[x,y]$ the geodesic connecting them.
Since $\Lambda$ is $\Gamma$-invariant,
the closed convex hull $H_c$ of $\Lambda$ in $X_c$ is $\Gamma$-invariant.
By \cite[Theorem 3.3]{Anderson83} or \cite[Corollary 2.5.3]{Bowditch95}, $H_c \cap X_\iota = \Lambda$.
If $|\Lambda|\ge2$, then
\begin{align}\label{covlim}
	H = H_c\cap X \ne \emptyset.
\end{align}
We may retract $X$ along the connecting geodesics $[\pi_H(x),x]$ onto $H$,
and this deformation retraction is $\Gamma$-equivariant.
We also obtain an induced orthogonal projection $\pi_C \colon O \to C$,
where $C$ denotes the \emph{convex core} of $O$, $C = C_\Gamma = \Gamma\backslash H$.
We get that $C$ is a deformation retract of $O$,
where the retraction is along the geodesics $[\pi_C(x),x]$.

\subsection{Parabolic groups, points, and ends}
\label{susepe}
We now explain the notion of parabolic ends.
Let $G$ be a parabolic group of isometries of $X$ with fix point $p\in X_\iota$.
Then $\Omega_G=X_\iota\setminus\{p\}$
and $M_c(G)=G\backslash(X\cup\Omega_G)$ has one end, \emph{the one coming from $p$}:
For any $x\in X$ and $\theta>0$, let
\begin{align}\label{cpx}
	C_p(x) = \cap_{g\in G}HC(gx,p,\theta)\setminus\{p\}
	\subseteq X_c\setminus\{p\} = X\cup\Omega_G,
\end{align}
where $HC(y,p,\theta)$ denotes the closed convex hull of the geodesic cone
in $X_c$ with apex at $y$, central direction $p$, and opening angle $\theta$.
Then
\begin{enumerate}
\item
$C_p(x)$ is a $G$-invariant, closed, and convex subset of $X\cup\Omega_G$;
\item
for any given sufficiently small $\theta>0$, the $G\backslash C_p(x)$, $x\in X$,
constitute a basis of neighborhoods of the unique end of $M_c(G)$.
\end{enumerate}
Compare with \cite[255:14--21]{Bowditch95}.

We say that a point $p\in X_\iota$ is a \emph{parabolic point of $\Gamma$}
if the stabilizer $G=\Gamma_p$ of $p$ in $\Gamma$ is a parabolic group such that,
for any $x\in X$ sufficiently close to $p$ and sufficiently small $\theta>0$, the set $C_p(x)$,
defined with respect to $\Gamma_p$, is \emph{precisely invariant};
that is, 
\begin{align}\label{precise}
	g\in \Gamma
	\hspace{2mm}\text{and}\hspace{2mm}
	gC_p(x) \cap C_p(x) \ne \emptyset
	\hspace{3mm}\Longrightarrow\hspace{3mm}
	g\in \Gamma_p, 
\end{align}
and then $gC_p(x)=C_p(x)$.
For any parabolic point $p$ of $\Gamma$,
$\Gamma_p\backslash C_p(x)$ embeds into $M_c(\Gamma)$
and the unique end of $\Gamma_p\backslash C_p(x)$ is an end of $O$, a \emph{parabolic end}.
Compare with \cite[264:1--20]{Bowditch95}.

\section{Adapted coverings and cutoff functions}
\label{secpu}

Given any discrete group $\Gamma$ of isometries of $X$, $\ve>0$, and $x\in X$,
let $\Gamma_\ve(x)$ be the subgroup of $\Gamma$
generated by the elements $g\in\Gamma$ with $d(x,gx)<\ve$.
For any subset $Y\subseteq X$, call
\begin{align}\label{teps}
	T_\ve(Y,\Gamma) = \{ x\in Y \mid |\Gamma_\ve(x)|=\infty\}
	\hspace{3mm}\text{and}\hspace{3mm}
	Y \setminus T_\ve(Y,\Gamma)
\end{align}
the \emph{$\ve$-thin} and \emph{$\ve$-thick part of $Y$} (with respect to $\Gamma$), respectively.
Recall that, by the Margulis lemma, $\Gamma_\ve(x)$ is virtually nilpotent if $0<\ve<\ve(m,a,b)$.
In what follows, we fix such an $\ve$. 

\subsection{Coverings of $X$ and $O$}
\label{subcov}
Let $O=\Gamma\backslash X$ be again a geometrically finite orbifold.
Let  $P$ be the set of parabolic points of $\Gamma$,
a $\Gamma$-invariant subset of $X_\iota$.
For any $p\in P$, set $U_p=T_\ve(X,\Gamma_p)$.

\begin{lem}\label{leme}
Let $p\in P$, $x\in U_p$, and $g\in\Gamma$.
\begin{enumerate}
\item\label{gx}
$\Gamma_\ve(x)$ is parabolic and contained in $\Gamma_p$.
\item\label{pq}
$U_p\cap U_q=\emptyset$ for all $q\ne p$ in $P$;
\item\label{gu}
$gU_p=U_{gp}$ and $gU_p\cap U_p\ne\emptyset$ implies that $g\in\Gamma_p$.
\end{enumerate}
\end{lem}

\begin{proof}
\eqref{gx}
By the Margulis Lemma, $\Gamma_\ve(x)$ is virtually nilpotent, hence elementary, by \cref{lemele}.  
Since $\Gamma_{p,\ve}(x)\subseteq\Gamma_\ve(x)$ and $\Gamma_{p,\ve}(x)$ is infinite,
$\Gamma_\ve(x)$ is not elliptic.
Furthermore, $\Gamma_{p,\ve}(x)$ is not loxodromix since it fixes $p$
and horosheres about $p$ as sets.
Hence $\Gamma_{p,\ve}(x)$ is parabolic,
therefore also $\Gamma_\ve(x)$ with fix point $p$.
In particular, $\Gamma_\ve(x)\subseteq\Gamma_p$.

\eqref{pq}
For $x\in U_p\cap U_q$, we have from \eqref{gx} that $\Gamma_\ve(x)$ is parabolic
and fixes $p$ and $q$. Hence $p=q$. 

\eqref{gu}
The first assertion is clear and the second follows immediately from the first in combination with \eqref{pq}.
\end{proof}

By the $\Gamma$-invariance of the families of $U_p$ and $H\cap U_p$,
they project to open subsets $V_p$ of $O=\Gamma\backslash X$
and $C\cap V_p$ of the convex core $C=\Gamma\backslash H$ of $O$, respectively.
For any $p\in P$,
\begin{align}\label{covp}
	\pi_H^{-1}(H\cap U_p) \to \pi_C^{-1}(C\cap V_p)
	= \Gamma_p\backslash\pi_H^{-1}(H\cap U_p)
\end{align}
is a Riemannian orbifold covering with the parabolic group $\Gamma_p$
as group of covering transformations.

Since $O$ has only finitely many parabolic ends, $\Gamma\backslash P$ is finite,
and hence there are only finitely many different $V_p$.
Moreover, by \cite[Proposition 4.1.2]{Bowditch95}, the
\begin{align}
	\pi_C^{-1}(C\cap V_p), \quad p \in \Gamma\backslash P,
\end{align}
are neighborhoods of the parabolic ends of $O$.
In particular,
\begin{align}\label{ceps}
	C_\ve = C\setminus\cup_{p\in\Gamma\backslash P} V_p
\end{align}
is compact.
The preimage of $C_\ve$ in $X$ is
\begin{align}\label{heps}
	H_\ve = H \setminus \cup_{p\in P} U_p.
\end{align}
Now choose metric balls $U_i=B(x_i,r_i)$ with $x_i\in H_\ve$, $i\in I$,
finite in number modulo $\Gamma$,
such that $gU_i\cap U_i\ne\emptyset$ for $g\in\Gamma$ implies
that $g$ belongs to the stabilizer $\Gamma_i=\Gamma_{x_i}$ of $x_i$ in $\Gamma$
and such that $H_\ve\subseteq\cup_{i\in I}U_i$.
Then $C_\ve\subseteq\cup_{i\in I}V_i$,
where $V_i$ is the image of $U_i$ in $O$.
By the setup, there are only finitely many different $V_i$.
For each $i$, $\pi_C^{-1}(C\cap V_i)$ is an open subset of $O$ and
\begin{align}\label{covi}
	\pi_H^{-1}(H\cap U_i) \to \pi_C^{-1}(C\cap V_i)
\end{align}
is a Riemannian orbifold covering with the finite group $\Gamma_i$ as group of covering transformations.

In conclusion, we obtain a $\Gamma$-invariant locally finite covering $\mathcal U$ of $X$
by the family $\pi_H^{-1}(H\cap U)$ of open subsets of $X$,
where $U$ runs over the $U_p$ and $U_i$
and a corresponding finite covering $\mathcal V$ of $O$
by the family $\pi_C^{-1}(C\cap V)$ of open subsets of $O$,
where $V$ runs over the $V_p$ and $V_i$.

\subsection{Cutoff functions}
\label{subcut}
Choose a $\Gamma$-invariant family $(\psi_U)_{U\in\mathcal U}$ of nonnegative smooth functions on $X$
such that $\supp\psi_U\subseteq U$ and such that the $\psi_U^2$ are a partition of unity
on $H$ subordinate to $\mathcal U$.
Set $\vf_U=\psi_U\circ\pi_H$.
Note that the $\vf_U$ may not be smooth, but are at least $C^{0,1}$ on $X\setminus H$.

By $\Gamma$-invariance, we obtain corresponding smooth functions $\psi_V$
such that $\supp\psi_V\subseteq V$ and $C^{0,1}$-functions $\vf_V=\psi_V\circ\pi_C$.

\begin{lem}\label{pu}
Given $\delta>0$, there is an $r>0$ such that $|\nabla\vf_U(x)|<\delta$
for all $U\in\mathcal U$ and $x\in X$, where $\pi_{H*x}$ exists and $d(x,H_\ve)>r$.
\end{lem}

\begin{proof}
By the setup, $U_0=\cup_{i\in I}U_i$ is a $\Gamma$-invariant neighborhood of $H_\ve$.
For any $p\in P$, $\vf_{U_p}=1$ on $\pi_H^{-1}(H\cap(U_p\setminus U_0))$.
In particular, all the $\vf_U$ are constant on $\cup_{p\in P}\pi_H^{-1}(U_p\setminus U_0)$
and, hence, their gradients vanish there.
On the other hand, there is an upper bound $C_0$ on the norm of the gradients of the $\psi_U$.
Furthermore, for any $x\in X\setminus H$, where $\pi_{H*x}$ exists,
vectors tangent to the minimal geodesic from $x$ to $H$ are in the kernel of $\pi_{H*x}$,
and, for any vector $u\in T_xX$ perpendicular to it, $|\pi_{H*x}u|\le|u|/\cosh(ar)$
by \cref{lemtec2} and \cref{curvest}, where $r=d(x,H)$.
\end{proof}

In the proof of \cref{sigmal}, we will also need that the functions $\nabla\vf_U$ are $C^2$
and that $\Delta\vf_U$ tends to $0$ uniformly as the distance to $H_\ve$ tends to infinity.
Since the above functions $\vf_U$ are, in general, only $C^{0,1}$,
we replace $H$ by a smooth convex domain
so that the corresponding new functions $\vf_U$ become smooth:
We note first that, for $f$ the distance function to $H$ and $\rho>0$,
any sublevel $\{f\le\rho\}$ is a strictly convex $C^{1,1}$-domain in $X$, by \cref{curvest2}.
Then we use \cref{papa} to replace $\{f\le\rho\}$ by a smooth and strictly convex domain $H'$ in $X$ with
\begin{align*}
	H \subseteq \{f\le\rho\} \subseteq H' \subseteq U_\eta(\{f\le\rho\}) = \{f\le\rho+\eta\}.
\end{align*}
By the setup and \cref{papa}, $H'$ is invariant under $\Gamma$.
For sufficiently small $\rho,\eta>0$, we have $H'\subseteq\cup_{U\in\mathcal U}U$.

Now we choose a $\Gamma$-invariant family $(\psi_U)_{U\in\mathcal U}$
of nonnegative smooth functions on $X$ such that $\supp\psi_U\subseteq U$
and such that the $\psi_U^2$ are a partition of unity on $H'$ subordinate to $\mathcal U$.
We also replace the orthogonal projection onto and the distance to $H$
by the corresponding projection onto and distance to $H'$.
Note that $H'\setminus\cup_{p\in P}U_p$ is compact modulo $\Gamma$
and that the distances to $H_\ve$ and to $H'\setminus\cup_{p\in P}U_p$
differ by at most $\rho+\eta$. 

\begin{lem}\label{pus}
Given $\delta>0$,
there is an $r>0$ such that \[|\nabla\vf_U(x)|,|\nabla^2\vf_U(x)|<\delta\]
for all $U\in\mathcal U$ and $x\in X$ with $d(x,H_\ve)>r$.
\end{lem}

The proof of the estimate of $\nabla\vf_U$ is the same as that in the proof of \cref{pu}.
The proof of the more sophisticated estimate of $\nabla^2\vf_U$ is contained in \cref{secjac2}.

\section{The main results}
\label{secray}

Let $\Gamma$ be a discrete group of isometries of $X$
and $O=\Gamma\backslash X$ be the associated orbifold quotient;
to start with, not necessarily geometrically finite.
Let $E$ be a Riemannian vector bundle over $O$
and $A$ be a formally self-adjoint elliptic differential operator on $E$ (of any order).
Lift $E$ and $A$ to $X$.
Assume that $A$ is essentially self-adjoint over $O$ and $X$.

In what follows, we use that $\lambda\in\R$ belongs to $\spec(A,O)$ if and only if
there is a sequence $(u_n)$ in $C^\infty_c(O,E)$ such that
\begin{align}\label{lamspec}
	\limsup\|u_n\|_2>0
	\hspace{3mm}\text{and}\hspace{3mm}
	\| Au_n-\lambda u_n\|_2 \to 0.
\end{align}
Moreover, $\lambda$ belongs to $\spec_{\ess}(A,O)$ if and only if there is such a sequence
which \emph{leaves any compact subset of $O$ eventually}, that is, such that,
for any compact subset $C$ of $O$,
\begin{align}\label{lamspece}
	\supp u_n \cap C = \emptyset
\end{align}
for all sufficiently large $n$.
By normalizing the $u_n$, one may also require that $\|u_n\|_2=1$ for all $n$.

We start with a remark, which seems to be well known;
compare with \cite[Theorem C]{CarronPedon04}.
Since we use it in our equality discussions, we review it shortly.

\begin{prop}\label{converse}
If $A$ is uniform and $\Omega_\Gamma\neq\emptyset$, then
\begin{align*}
	\spec(A,X)=\spec_{\ess}(A,X)\subseteq\spec_{\ess}(A,O).
\end{align*}
\end{prop}

\begin{proof}
Let $x\in\Omega_\Gamma$ be a point with trivial isotropy group.
Choose an open neighborhood $U$ of $x$ in $X_c$
such that $\gamma U\cap U\ne\emptyset$ for $\gamma\in\Gamma$
implies that $\gamma=1$.
Let $U\supseteq U_1\supseteq U_2\supseteq\dots$
be a nested sequence of neighborhoods of $x$ in $X_c$ such that $\cap U_n=\{x\}$.

Given $\lambda\in\spec(A,X)$, choose a sequence of $u_n$ as in \eqref{lamspec}.
Then there are $g_n\in G$ such that, for each $n$,
the support of $v_n=g_nu_ng_n^{-1}$ is contained in $U_n$.
By the assumption on $U$, the $v_n$ can be pushed down to $E$ over $O$.
By the assumption on the $U_n$,
the supports of the $v_n$ and the pushed down $v_n$ leave any compact subset
of $X$ respectively $O$ eventually.
Hence $\lambda$ belongs to $\spec_{\ess}(A,X)$ and $\spec_{\ess}(A,O)$.
\end{proof}

We assume from now on that $O=\Gamma\backslash X$ is geometrically finite.
We start with the case where $\Gamma$ is elementary, that is, $|\Lambda|\le2$,
where the arguments apply to all essentially self-adjoint differential operators $A$.

\begin{prop}\label{elementary}
If $\Gamma$ is elementary, then
\begin{align*}
	\spec(A,O) \subseteq \spec(A,X)
	\hspace{3mm}\text{and}\hspace{3mm}
	\spec_{\ess}(A,O) \subseteq \spec_{\ess}(A,X).
\end{align*}
Moreover, if $A$ is uniform, then $\spec_{\ess}(A,X)=\spec_{\ess}(A,O)$.
\end{prop}

\begin{proof}
Since $\Omega_\Gamma\ne X_\iota$ in each case,
the last assertion follows immediately from \cref{converse},
once the second asserted inclusion is established. 
Thus it suffices now to prove the asserted inclusions.

\textit{Case 0:}
Assume first that $|\Lambda|=0$ or, equivalently, that $\Gamma$ is finite.
It is well known that lifting sections from $O$ to $X$ leads to an isomorphism between $A$
on sections of $E$ over $O$ and $A$ on $\Gamma$-invariant sections of $E$ over $X$.
Thus the spectrum and essential spectrum of $A$ on $E$ over $O$ are actually the same
as the spectrum and essential spectrum of $A$
on the subspcae of $\Gamma$-invariant sections of $E$ over $X$.

\textit{Case 1:}
Assume now that $|\Lambda|=1$ or, equivalently, that $\Gamma$ is parabolic.
Then $\Gamma$ is a finitely generated and virtually nilpotent group
with unique fix point $p\in X_\iota$ and $\Omega=X_\iota\setminus\{p\}$.
In particular, the covering $X\to O$ is infinite and amenable.
Let $N$ be a normal, finitely generated, and torsion free nilpotent subgroup of $\Gamma$.
Then the orbifold covering $X\to O$ decomposes into the covering $X\to N\backslash X$ of manifolds
and the finite normal orbifold covering $N\backslash X\to O$
with group $\Gamma/N$ as group of covering transformations.
To the latter we apply the arguments from the first case, substituting  $N\backslash X$ for $X$,
to obtain that $\spec(A,O)\subseteq \spec(A,N\backslash X)$
and $\spec_{\ess}(A,O) \subseteq \spec_{\ess}(A,N\backslash X)$.
This reduces the assertions to the case of the amenable covering $X\to N\backslash X$ of manifolds,
where they follow from \cref{pp}.

\textit{Case 2:}
Assume finally that $|\Lambda|=2$.
Then $\Gamma$ is a finite extension of a normal infinite cyclic subgroup $Z\subseteq\Gamma$
such that $\Gamma$ fixes a geodesic $[p,q]\in X_c$ as a set.
As in the previous case, we may pass now to the finite cover $Z\backslash X$ of $O$;
that is, we may assume that $O=Z\backslash X$.
But $Z$ is amenable and hence the assertions follow again from \cref{pp}, as in the previous case.
\end{proof}

In the proof of \cref{elementary}, Lemmas \ref{pu} and \ref{pus} do not play a role.
The latter will be used to reduce the case of general geometrically finite orbifolds
to situations similar to the ones considered above,
and it is this reduction, where we need assumptions on the operators.

In the proofs of Theorems \ref{bottom} and \ref{sigmaf},
let $\mathcal U$ and $\mathcal V$ be locally finite respectively finite open coverings
of $X$ and $O$ and $(\vf_U)_{U\in\mathcal U}$ and $(\vf_V)_{V\in\mathcal V}$
be associated families of functions in $C^{0,1}(X)$ and $C^{0,1}(O)$ as in \cref{pu}.

Let $E$ be a Riemannian vector bundle over $O$ together with a metric connection $\nabla$. 
Let $A$ be a formally self-adjoint Laplace type operator on $E$ over $O$.
Then $B=A-\Delta$ is a formally self-adjoint differential operator of first order,
where $\Delta=\nabla^*\nabla$.
For $u\in C^\infty_c(O,E)$, we obtain
\begin{align*}
	\la Au,u\ra_2 &- \sum_{V\in\mathcal V} (\|\nabla(\vf_Vu)\|_2^2 + \la B(\vf_Vu),\vf_Vu\ra_2) \\
	&= \|\nabla u\|_2^2 + \la Bu,u\ra_2 
	- \sum_{V\in\mathcal V} (\|\nabla(\vf_Vu)\|_2^2 + \la B(\vf_Vu),\vf_Vu\ra_2) \\
	&= \int \big\{ (|\nabla u|^2+\la Bu,u\ra)(1-\sum_{V\in\mathcal V} \vf_V^2)
	- \sum_{V\in\mathcal V}|\nabla\vf_V|^2|u|^2 \\
	&\hspace{5mm}- 2\sum_{V\in\mathcal V} \la\nabla\vf_V\otimes u,\vf_V\nabla u\ra 
	- \sum_{V\in\mathcal V} \la\sigma_B(\nabla\vf_V)u,\vf_Vu\ra \big\},
\end{align*}
where $\sigma_B$ denotes the principal symbol of $B$.
Since $\sum\vf_V^2=1$, the first integrand on the right vanishes.
The third and fourth integrand equal
\begin{align*}
	2\sum_{V\in\mathcal V} \la\vf_V\nabla\vf_V\otimes u,\nabla u\ra
	= \sum_{V\in\mathcal V} \la(\nabla\vf_V^2)\otimes u,\nabla u\ra
\end{align*}
respectively
\begin{align*}
	\sum_{V\in\mathcal V} \la\vf_V\sigma_B(\nabla\vf_V)u,u\ra
	= \frac12 \sum_{V\in\mathcal V} \la\sigma_B(\nabla\vf_V^2)u,u\ra
\end{align*}
and therefore vanish, again since $\sum\vf_V^2=1$.
In conclusion,
\begin{align}\label{partition}
	\sum_{V\in\mathcal V} (\|\nabla(\vf_Vu)\|_2^2 + \la B(\vf_Vu),\vf_Vu\ra_2)
	= \la Au,u\ra_2 + \sum_{V\in\mathcal V} \int |\nabla\vf_V|^2|u|^2.
\end{align}
For $u\ne0$, we obtain
\begin{align}\label{estray}
\begin{split}
	&\frac{\sum_{V \in \mathcal V} (\|\nabla(\vf_Vu)\|_2^2
	+ \la B(\vf_Vu),\vf_Vu\ra)_2}{\sum_{V \in \mathcal V} \|\vf_Vu\|_2^2} \\
	&\hspace{25mm}= \frac{\la Au,u\ra_2
	+ \sum_{V\in\mathcal V} \int|\nabla\vf_V|^2|u|^2}{\sum_{V\in\mathcal V} \int \vf_V^2|u|^2} \\
	&\hspace{25mm}= \frac{\la Au,u\ra_2 + \sum_{V\in\mathcal V} \int|\nabla\vf_V|^2|u|^2}{\|u\|_2^2} \\
	&\hspace{25mm}\le \frac{\la Au,u\ra_2}{\|u\|_2^2} + \sum_{V\in\mathcal V} \|\nabla\vf_V\|_{\supp u,\infty}^2.
\end{split}
\end{align}
In particular, if $u\ne0$, then there is a $V\in\mathcal V$ with $\vf_Vu\ne0$ such that
\begin{align}\label{estray2}
	\Ray_A(\vf_Vu)
	\le \Ray_A(u) + \sum_{V\in\mathcal V} \|\nabla\vf_V\|_{\supp u,\infty}^2,
\end{align}
where $\Ray$ indicates the Rayleigh quotient with respect to $A$.
Note that we use the Rayleigh quotients of $A$ here in the form
\begin{align*}
	\Ray_A(v) = \frac{\|\nabla v\|_2^2 + \la Bv,v\ra_2}{\|v\|_2^2}
\end{align*}
for $C^{0,1}$-sections of $E$.
This is legitimate since the infimum over them yield the bottom of the spectrum of the closure of $A$,
as is well known.

In the spirit of the above discussion, we also get the following criterion.

\begin{lem}\label{minfty}
If $A$ is not bounded from below on $E$ over $O$,
then there is a sequence $(u_n)$ in $C^\infty_c(O,E)$,
which leaves any compact subset of $O$ eventually,
such that $\|u_n\|_2=1$ and $\Ray_A(u_n)\to-\infty$.
\end{lem}

\begin{proof}
Suppose that $A$ is not bounded from below on $E$ over $O$.
Then there is a sequence $(u_n)$ in $C^\infty_c(O,E)$
such that $\|u_n\|_2=1$ and $\Ray_A(u_n)\to-\infty$.
Let $\vf\in C^\infty(O)$ with $0\le\vf\le1$ such that $\psi=(1-\vf^2)^{1/2}$ is smooth
and such that $\supp\vf$ is a compact domain with smooth boundary.
Computing as above, we obtain that
\begin{align*}
	\min\{\Ray_A(\vf u_n),\Ray(\psi u_n)\} \to -\infty.
\end{align*}
Since $A$ is bounded from below on $\supp\vf$ with Dirichlet boundary condition,
we get that $\Ray(\psi u_n)\to -\infty$.
Thus modifying the sequence of $(u_n)$ appropriately,
it leaves any compact subset of $O$ eventually.
\end{proof}

\begin{proof}[Proof of \cref{bottom}]
Let $(u_n)$ be a sequence in $C^\infty_c(O,E)$, which leaves any compact subset of $O$ eventually,
such that $\|u_n\|_2=1$ and $\Ray_A(u_n)\to-\infty$ if $A$ is not bounded from below
or such that $\Ray_A(u_n)\to\lambda_{\ess}(A,O)$ otherwise.
Then
\begin{align*}
	\sum_{V\in\mathcal V} \|\nabla\vf_V\|_{\supp u_n,\infty}^2 \to 0,
\end{align*}
by \cref{pu}.
For each $n$, choose a $V_n\in\mathcal V$ such that $\vf_{V_n}u_n\ne0$
and such that $\Ray_A(\vf_{V_n}u_n)$ satisfies the estimate in \eqref{estray2}.
By passing to a subsequence if necessary, we can assume that $V_n$ does not depend on $n$,
but is a fixed $V\in\mathcal V$.
Then $(\vf_Vu_n)$ is a sequence in $C^\infty_c(O,E)$ with $\vf_Vu_n\ne0$
and $\supp(\vf_Vu_n)\subset\pi_C^{-1}(C\cap V)\cap\supp u_n$,
so that $(\vf_Vu_n)$ leaves any compact subset of $O$ eventually,
and such that $\limsup\Ray_A(\vf_Vu_n)\le\lim\Ray_A(u_n)$.

There are now two cases.
Either $V=V_i$ and belongs to the finite subfamily of $\mathcal V$ covering $C_\ve$.
Then we have the covering $\pi_H^{-1}(H\cap U_i)\to\pi_C^{-1}(C\cap V_i)$,
where $U\in\mathcal U$ lies above $V$
and the group $\Gamma_i$ of covering transformations is finite.
Then we argue as in the first case in \cref{elementary},
substituting $\pi_H^{-1}(H\cap U_i)$ for $X$ and $\pi_C^{-1}(C\cap V_i)$ for $O$.
The other case is that $V=V_p$ for some parabolic point $p$ of $\Gamma$.
Then we have the covering $\pi_H^{-1}(H\cap U_p)\to\pi_C^{-1}(C\cap V_p)$
with  group $\Gamma_p$ of covering transformations.
Since $\Gamma_p$ is finitely generated and virtually nilpotent,
we can argue as in the second case in \cref{elementary},
substituting $\pi_H^{-1}(H\cap U_p)$ for $X$ and $\pi_C^{-1}(C\cap V_p)$ for $O$.
The equality assertion follows from \cref{converse},
where we note that the volume of geometrically finite orbifolds $O$ is infinite
if and only if $\Omega_\Gamma$ is not empty.
\end{proof}

Let now $A$ be of first order, $\lambda\in\R$, and $u\in C_c^\infty(O,E)$.
Then
\begin{align}\label{partition1}
\begin{split}
	\sum_{V\in\mathcal V} |A(\vf_Vu)&-\lambda\vf_Vu|^2
	= \sum_{V\in\mathcal V} |\sigma_A(\nabla\vf_V)u + \vf_V(Au-\lambda u)|^2 \\
	&\le 2\sum_{V\in\mathcal V} |\sigma_A(\nabla\vf_V)u|^2
	+ 2\sum_{V\in\mathcal V} \vf_V^2|Au-\lambda u|^2 \\
	&= 2\sum_{V\in\mathcal V} |\sigma_A(\nabla\vf_V)|^2 |u|^2 + 2|Au-\lambda u|^2 \\
	&\le 2\sum \|\sigma_A\|_\infty\|\nabla\vf_V\|_{\supp u,\infty}^2 |u|^2 + 2|Au-\lambda u|^2.
\end{split}
\end{align}

\begin{proof}[Proof of \cref{sigmaf}]
Let $\lambda$ be in the essential spectrum of $A$
and $(u_n)_{n\in\N}$ be a sequence in $C^\infty_c(O,E)$ as in \eqref{lamspec} and \eqref{lamspece}.
Then by \eqref{partition1}, after passing to a subsequence if necessary,
there is a $V\in\mathcal V$ such that the sequence of $\vf_Vu_n$ also satisfies
the requirements of \eqref{lamspec} and \eqref{lamspece},
except for the smoothness of the sections.
However, the $\vf_Vu_n$ are still in the domain of the closure of $A$,
which is a sufficient replacement for smoothness in \eqref{lamspec} and \eqref{lamspece}.
Thus we arrive at the first assertion, arguing as in the proof of \cref{bottom}.
The second assertion follows immediately from \cref{converse},
observing again that the volume of $O$ is infinite
if and only if $\Omega_\Gamma$ is not empty.
\end{proof}

Assume now again that $A$ is of Laplace type and write $A=\Delta+B$, where $B$ is of first order.
Since we need the sections $\vf_V u$ to be in the domain of the closure of $A$,
we assume now that $\mathcal U$ and $\mathcal V$ are locally finite respectively finite open coverings
of $X$ and $O$ and that $(\vf_U)_{U\in\mathcal U}$ and $(\vf_V)_{V\in\mathcal V}$
are associated families of smooth functions as in \cref{pus}.
Then
\begin{align*}
	\sum_{V\in\mathcal V} &|(A-\lambda)(\vf_Vu)|^2
	= \sum_{V\in\mathcal V} |\Delta(\vf_Vu) + B(\vf_Vu) - \lambda\vf_Vu|^2 \\
	&= \sum_{V\in\mathcal V} |\vf_V(\Delta u + Bu - \lambda u) \\
	&\hspace{15mm} -2 \tr(\nabla\vf_V\otimes\nabla u) + (\Delta\vf_V)u + \sigma_B(\nabla\vf_V)u|^2 \\
	&\le 4 \sum_{V\in\mathcal V} \{\vf_V^2|\Delta u + Bu - \lambda u|^2
	+ 4 \|\nabla\vf_V\|_{\supp u,\infty}^2|\nabla u|^2 \} \\
	&\hspace{15mm}+ 4 \sum_{V\in\mathcal V}(\|\Delta\vf_V\|_{\supp u,\infty}^2
	+ \|\sigma_B\|_{\infty}^2\|\nabla\vf_V\|_{\supp u,\infty}^2)|u|^2.
\end{align*}
Using that $\sum\vf_V^2=1$, we obtain
\begin{align}\label{partition2}
\begin{split}
	\sum_{V\in\mathcal V} |(A-\lambda)&(\vf_Vu)|^2
	\le 4 |(A-\lambda)u|^2 + 16 \sum_{V\in\mathcal V} \|\nabla\vf_V\|_{\supp u,\infty}^2|\nabla u|^2 \\
	&+ 4 \sum_{V\in\mathcal V} (\|\Delta\vf_V\|_{\supp u,\infty}^2
	+ \|\sigma_B\|_{\infty}^2\|\nabla\vf_V\|_{\supp u,\infty}^2)|u|^2,
\end{split}
\end{align}
where $\sigma_B$ denotes the principal symbol of $B$.

\begin{proof}[Proof of \cref{sigmal}]
With respect to a local orthonormal frame $(X_i)$ of $O$,
\begin{align*}
	|\nabla u|^2 + \la Bu,u\ra
	&= |\nabla u|^2 + \sum \la \sigma_B(X_i)\nabla_{X_i}u,u\ra + \la Vu,u\ra \\
	&\ge |\nabla u|^2 - |\sigma_B||\nabla u||u| + \la Vu,u\ra \\
	&\ge |\nabla u|^2 - \alpha\|\sigma_B\|_\infty|\nabla u|^2
	- (\|\sigma_B\|_\infty/\alpha)|u|^2 + \la Vu,u\ra \\
	&\ge - (c_0+ \|\sigma_B\|_\infty/\alpha)|u|^2,
\end{align*}
where we choose $\alpha\le1/\|\sigma_B\|_\infty$ and where $-c_0$ is a lower bound for $V$.
Therefore
\begin{align*}
	\la Au,u\ra_2 \ge - (c_0+ \|\sigma_B\|_\infty/\alpha)\|u\|_2^2,
\end{align*}
and hence $A$ is bounded from below.
Thus $A$ is also essentially self-adjoint, by \cite[Theorem A.24]{BallmannPolymerakis20a}.

The proof of the assertions about the spectrum is similar to that of \cref{sigmaf};
however, the control of the terms on the right in \eqref{partition2} is different.
We let $(u_n)$ be a sequence in $C^\infty_c(O,E)$ for $\lambda\in\spec_{\ess}(A,O)$
as in \eqref{lamspec} and \eqref{lamspece}.
Then $\la Au_n,u_n\ra_2\le(|\lambda|+1)\|u_n\|_2^2$ for all sufficiently large $n$
and we obtain, for any such $u=u_n$, 
\begin{align*}
	\|\nabla u\|_2^2
	&= \la Au,u\ra_2 - \la Bu,u\ra_2 \\
	&\le (|\lambda|+1)\|u\|_2^2	 - \la Bu,u\ra_2 \\
	&= (|\lambda|+1)\|u\|_2^2 - \sum \la B_i\nabla_{X_i}u,u\ra_2 - \la Vu,u\ra_2 \\
	&\le (|\lambda|+1+c_0)\|u\|_2^2 + c_1\|\nabla u\|_2\|u\|_2 \\
	&\le  (|\lambda|+1+c_0+c_1/\alpha)\|u\|_2^2 + \alpha c_1\|\nabla u\|_2^2,
\end{align*}
where $-c_0$ is a lower bound for $V$ and $c_1=\|\sigma_B\|_\infty$.
By choosing an $\alpha<1/c_1$, we get an estimate for $\|\nabla u\|_2^2$.
The remaining part of the proof is now analogous to that of \cref{sigmaf}.
\end{proof}

\section{Estimates for derivatives of Jacobi fields}
\label{secjac2}

This section is devoted to the proof of the second estimate in \cref{pus}.
Let $C$ be a strictly convex  domain in $X$ with smooth boundary.
Then the distance $f$ to $C$ and the orthogonal projection $\pi$ onto $C$
are smooth on $X\setminus C$.
Let $\psi\colon C\to\R$ be a smooth function and set $\vf=\psi\circ\pi$.
Our aim is to estimate $\Delta\vf$.
To that end, estimates of derivatives of Jacobi fields in the variational direction will be crucial.

Let $c=c_s(t)=c(s,t)$ be a geodesic variation of a geodesic $c_0$ by geodesics $c_s$,
defined on $(-\ve,\ve)\times[0,r]$, and assume that all geodesics are of unit speed.
Then we have the velocity fields $c'=\partial c/\partial t$ and Jacobi fields $J=\partial c/\partial s$,
and our aim is to estimate the component $K^\perp$ perpendicular to $c$ of
\begin{align*}
	K = \frac{\nabla J}{\partial s} = \frac{\nabla}{\partial s}\frac{\partial c}{\partial s}.
\end{align*}
To that end, we assume that $J$ is perpendicular to $c'$.
Now
\begin{align}\label{j}
\begin{split}
	\frac{\nabla}{\partial s} R(J,c')c'
	&= \nabla R(J,J,c')c' + R(K,c')c' + R(J,J')c' + R(J,c')J'.
\end{split}
\end{align}
Recalling the Jacobi equation $J''+R(J,c')c'=0$, we compute
\begin{align}\label{kr}
\begin{split}
	- \frac{\nabla}{\partial s} R(J,c')c'
	&= \frac{\nabla}{\partial s}J'' \\
	&= \frac{\nabla}{\partial s}\frac{\nabla}{\partial t}\frac{\nabla}{\partial t}\frac{\partial c}{\partial s} \\
	&= \frac{\nabla}{\partial t}\frac{\nabla}{\partial s}\frac{\nabla}{\partial t}\frac{\partial c}{\partial s}
	+ R(\frac{\partial c}{\partial s},\frac{\partial c}{\partial t})\frac{\nabla}{\partial t}\frac{\partial c}{\partial s} \\
	&= \frac{\nabla}{\partial t}\frac{\nabla}{\partial t}\frac{\nabla}{\partial s}\frac{\partial c}{\partial s}
	+ \frac{\nabla}{\partial t}R(\frac{\partial c}{\partial s},\frac{\partial c}{\partial t})\frac{\partial c}{\partial s}
	+ R(\frac{\partial c}{\partial s},\frac{\partial c}{\partial t})\frac{\nabla}{\partial t}\frac{\partial c}{\partial s} \\
	&= K'' + \nabla R(c',J,c')J + R(J',c')J + 2R(J,c')J',  
\end{split}
\end{align}
where we use $c''=0$.
With \eqref{j}, \eqref{kr}, and the Bianchi identity for $R$,
we obtain the perturbed Jacobi equation
\begin{align}\label{jk}
	K'' + R(K,c')c'
	= 4R(c',J)J' + \nabla R(c',c',J)J - \nabla R(J,J,c')c'.
\end{align}
Note that the last two terms on the right vanish in the case of hyperbolic spaces
and their quotients.

Let $x\in X\setminus C$ and $r=f(x)>0$.
Let $c_2,\dots,c_m$ be unit speed geodesics in $\{f=r\}$ through $x$,
which meet orthogonally in $x$.
Using that $\psi$ does not depend on $r$ and that $c_i''$ is perpendicular to $\{f=r\}$ for $i\ge2$,
we get
\begin{align}\label{delta}
\begin{split}
	&\Delta\psi(x)
	= -\sum_{2\le i\le m} (\vf\circ\pi\circ c_i)''(0)
	= -\sum_{2\le i\le m} (d\vf\circ(\pi\circ c_i)')'(0) \\
	&= -\sum_{2\le i\le m} \big\{ (\nabla^2\vf((\pi\circ c_i)'(0),(\pi\circ c_i)'(0))
	+ d\vf((\pi\circ c_i)''(0))\big\},
\end{split}
\end{align}
where the prime indicates derivatives with respect to the parameter $s$ of the $c_i$.
Going back to the situation in \eqref{kr}, if $i\ge2$ is given
and $c_s=c_{i,s}$ is the unit speed geodesic from $\pi(c_i(s))$ to $c_i(s)$, then
\begin{align*}
	(\pi\circ c_i)'(0) = J(0,0)
	\hspace{3mm}\text{and}\hspace{3mm}
	(\pi\circ c_i)''(0) = K(0,0).
\end{align*}
Since $\nabla^2\vf$ is uniformly bounded on compact subsets of $C$
and since $|J(0,0)|$ decays exponentially as $r\to\infty$,
the first term on the right in \eqref{delta} is under good control.
It remains to show that $d\vf(K(0,0))\to0$ in a controlled way as $r\to\infty$.
The following approach was motivated by \cite[§7]{HopfE40}.

For convenience, let now $J=J(0,.)$ and $K=K(0,.)$.
Let $L$ be a Jacobi field along $c=c_0$ perpendicular to $c$ and such that $L(r)=0$.
Then
\begin{align}\label{fl}
	(\la K',L\ra - \la K,L'\ra)' = \la F,L\ra
\end{align}
where $F$ denotes the right hand side of \eqref{jk}.
Now $K^\perp(r)=0$ by the above choice of variation and $L(r)=0$, hence
\begin{align}\label{fl2}
	\la K,L'\ra(0) -  \la K',L\ra(0) = \int_0^r \la F,L\ra\dt.
\end{align}
Comparison with constant curvature $-a^2$ gives
\begin{align}\label{le}
\begin{split}
	\frac{|L|'}{|L|}(t) \le& - a\coth(a(r-t))   \\
	\frac{|L(t)|}{|L(0)|} \le& \frac{\sinh(a(r-t))}{\sinh(ar)}.
\end{split}
\end{align}
The second fundamental form $S_0$ of $C=C_0$ is positive definite
and, on compact parts of $C_0$, bounded between constants $0<\kappa_-\le\kappa_+$.
By \Cref{curvest2}, the second fundamental form $S_t$ of $\{f=t\}$ satisfies
\begin{align*}
	\frac{j_-'}{j_-} \le S_t \le \frac{j_+'}{j_+},
\end{align*}
where
\begin{align}\label{jmp}
\begin{split}
	j_- = \cosh(at)+\frac{\kappa_-}{a}\sinh(at)
	\hspace{3mm}\text{and}\hspace{3mm}
	j_+ = \cosh(bt)+\frac{\kappa_+}{b}\sinh(bt).
\end{split}
\end{align}
Since $J'(t)=S_tJ(t)$, we obtain
\begin{align}\label{jprnorm}
	\frac{j_-'}{j_-}
	\le \frac{\la J',J\ra}{\la J,J\ra}
	= \frac{|J|'}{|J|} \le \frac{|J'|}{|J|}
	\le \frac{j_+'}{j_+}.
\end{align}
Therefore, since $|J(r)|=1$ by our setup,
\begin{align}\label{je}
	\frac{j_-}{j_-(r)}
	\ge |J|
	\ge \frac{j_+}{j_+(r)}.
\end{align}
In particular,
\begin{align}\label{jest0}
	|J(0)| \le \frac1{\cosh(ar)+\kappa_-\sinh(ar)/a}.
\end{align}
To estimate the integral in \eqref{fl2}, assuming that $\|\nabla R\|_\infty<\infty$,
we invoke \eqref{jk}, \eqref{le}, \eqref{jprnorm}, and \eqref{je} to get
\begin{align}\label{estint}
\begin{split}
	\int_0^r |F||L|\dt
	&\le C \int_0^r \frac{j_-^2}{j_-(r)^2}\frac{\sinh(a(r-t))}{\sinh(ar)}\dt \\
	&= C \int_0^r \frac{(a\cosh(at)+\kappa_-\sinh(at))^2}{(a\cosh(ar)+\kappa_-\sinh(ar))^2}
	\frac{\sinh(a(r-t))}{\sinh(ar)}\dt \\
	&\le C' \int_0^r \frac{e^{2at}}{e^{2ar}}\frac{e^{a(r-t)}}{e^{ar}}\dt
	= C' \frac{1}{e^{2ar}} \int_0^r e^{at}\dt \le C'' e^{-ar}.
\end{split}
\end{align}
Since $L$ is perpendicular to $c$, only the component $K^\perp$ of $K$ perpendicular to $c$
comes into play in \eqref{fl} and \eqref{fl2}, where we note that $(K')^\perp=(K^\perp)'$.
We compute
\begin{align}\label{kprime2}
	\begin{split}
		(K^\perp)^\prime
		&= \left( \frac{\nabla}{\partial t}\frac{\nabla}{\partial s}\frac{\partial c}{\partial s} \right)^\perp
		= \left( \frac{\nabla}{\partial s}\frac{\nabla}{\partial t}\frac{\partial c}{\partial s}
		+ R(c^\prime,J)J \right)^\perp \\
		&= \left( \frac{\nabla}{\partial s}\frac{\nabla}{\partial s}\frac{\partial c}{\partial t} \right)^\perp
		+ \left(R(c^\prime,J)J \right)^\perp \\
		&= (\nabla_{J} \nabla_J c^\prime)^\perp + \left(R(c^\prime,J)J \right)^\perp \\
		&= (\nabla_{J} S_t J)^\perp + \left(R(c^\prime,J)J \right)^\perp \\
		&= \nabla^t_J(S_tJ) + \left(R(c^\prime,J)J \right)^\perp \\
		&= (\nabla^t S_t)(J,J) + S_tK^\perp + \left(R(c^\prime,J)J \right)^\perp,
	\end{split}
\end{align}
where $S_t$ is the second fundamental form of $\{f=t\}$ with respect to $c'(t)$
and $\nabla^t$ stands for the Levi-Civita connection of $\{f=t\}$.

To control the left hand side of \eqref{fl2}, we will need \eqref{kprime2} for $t=0$.
Since we only consider a compact part of $C$, $\nabla^0S_0$ is bounded uniformly over that part.
Hence, by \eqref{jest0},
\begin{align}\label{estrem}
	|(\nabla^0S_0)(J,J) + \left(R(c^\prime,J)J \right)^\perp|
	\le \frac{C}{(\cosh(ar)+\kappa_-\sinh(ar)/a)^2}
\end{align}
Hence, by \eqref{estint} and \eqref{estrem},
\begin{align*}
	|\la K^\perp,L'\ra(0) - \la S_0K^\perp,L\ra(0)| \le C e^{-ar}
\end{align*}
On the other hand, $L'(0)=B(0)L(0)$, where $B$ is the second fundamental form of the sphere of radius $r$
about $c(r)$ at $c(0)$ with respect to the inner normal $c'(0)$.
Now $B<0$ and $S_0>0$ and
\begin{align*}
	\la K^\perp,L'\ra(0) - \la S_0K^\perp,L\ra(0)
	= \la (B-S_0)K^\perp,L\ra(0)
\end{align*}
for any Jacobi field $L$ as above.
Since $B-S_0<-S_0<0$, we conclude that $|K^\perp(0)|\le Ce^{-ar}$.

\appendix

\section{Smoothing convex sets}
\label{secmoll}

The purpose of this appendix is the proof of the following version
of \cite[Proposition 6]{ParkkonenPaulin12} of Parkkonen and Paulin.
The setting is a closed and strictly convex $C^{1,1}$-domain $C\subseteq X$
such that the second fundamental form $S$ of $\partial C$
satisfies $0<\alpha\le S\le\beta$ almost everywhere.

\begin{thm}\label{papa}
Assume that $\nabla R$ is uniformly bounded in a tubular neighborhood $U_\rho(C)$ about $C$,
for some $\rho>0$, and let $0<\eta<\alpha\wedge\rho$.
Then there is a closed and strictly convex domain $C\subseteq C'\subseteq U_\eta(C)$
whose boundary is smooth with second fundamental form $\alpha-\eta\le S'\le\beta+\eta$.
Moreover, $C'$ is invariant under any isometry of $X$ leaving $C$ invariant.
\end{thm}

Note that the statement of \cite[Proposition 6]{ParkkonenPaulin12} is more general,
since there is no assumption on $\nabla R$.
But it seems that the arguments behind the last eight lines on \cite[Page 630]{ParkkonenPaulin12}
require the compactness of $C$.

As in \cite{ParkkonenPaulin12}, our proof relies on smoothing the distance function to $C$.
We will need uniform estimates on the first and second derivative of the smoothed distance function.
Our arguments also work in more general situations.
However, in our setting, the arguments are less cumbersome.

Let $F\colon X\times\R^m\to TX$ be a smooth orthonormal frame of $X$ and write
\begin{align}\label{gexp}
	g_v(x) = g(x,v) = \exp F(x,v).
\end{align}
Note that $g_0$ is the identity of $X$ and that, therefore,
$g_v$ converges locally uniformly in the $C^\infty$-topology to the identity of $X$ as $|v|\to0$.
This is the point of the arguments in \cite{ParkkonenPaulin12}, which we alluded to above.

Let $\vf\colon\R\to\R$ be a non-negative smooth function which is positive and constant
in a neighborhood of $0$ such that $\vf(r)=0$ for $r\ge1$ and such that
\begin{align*}
	\int_{\R^m}\vf(|v|)\dv = 1.
\end{align*}
For $\kappa>0$, set 
\begin{align*}
	\vf_\kappa = \vf_\kappa(x) = \kappa^{-m}\vf(x/\kappa).
\end{align*}
Let $f$ be a locally integrable function on $X$.
Then
\begin{align}\label{fk2}
	f_\kappa(x) = \int_{\R^m} \vf_\kappa(|v|)f(g(x,v)) \dv
\end{align}
does not depend on the choice of the frame $F$.
By the usual rule of differentiation under the integral sign, we obtain that $f_\kappa$ is smooth.

\begin{lem}\label{invari}
For any isometry $\gamma$ of $X$ and $\kappa>0$,
$(f\circ\gamma)_\kappa=f_\kappa\circ\gamma$.
In particular, if $f$ is $\gamma$-invariant, then also $f_\kappa$.
\end{lem}

\begin{proof}
For a given choice $F$ of frame, define a new frame $\gamma^*F$ by
\begin{align*}
	\gamma^*F(x,v) = \gamma_*F(\gamma^{-1}x,v).
\end{align*}
Then $\gamma^*F$ is also a frame, and we get
\begin{align*}
	(f\circ\gamma)_\kappa(x)
	&= \int_{\R^m} \vf_\kappa(|v|)f(\gamma\exp F(x,v)) \dv \\
	&= \int_{\R^m} \vf_\kappa(|v|)f(\exp\gamma_*F(x,v)) \dv \\
	&= \int_{\R^m} \vf_\kappa(|v|)f(\exp\gamma^*F(\gamma x,v)) \dv
	= f_\kappa(\gamma x),
\end{align*}
where we use that $f_\kappa$ does not depend on the choice of frame.
\end{proof}

Suppose now that $f$ is $C^{0,1}$.
Then $f$ is differentiable almost everywhere
and the norm of its derivative is locally integrable.
Differentiation under the integral sign gives
\begin{align}\label{dfk}
	\nabla f_\kappa|_x(w)
	= \int_{\R^m} \vf_\kappa(|v|)\nabla(f\circ g_v)|_x(w)\dv
\end{align}
for all $x\in X$ and $w\in T_xX$, where $g_v(x)=g(x,v)$.
More generally, we have

\begin{lem}\label{d2fk}
For any $k\ge1$, if $f$ is $C^{k-1,1}$, then
\begin{align*}
	\nabla^kf_\kappa|_x(w_1,\dots,w_k)
	= \int_{\R^m} \vf_\kappa(|v|)\nabla^k(f\circ g_v)|_x(w_1,\dots,w_k)\dv
\end{align*}
for all $x\in X$ and $w_1,\dots,w_k\in T_xX$.
\end{lem}

\begin{proof}
For convenience, we write out the formulas only in the case $k=2$.
Choose a smooth vector field $W_2$ in a neighborhood of $x$ such that $W_2(x)=w_2$.
Then we have
\begin{align*}
	\nabla^2f_\kappa|_x&(w_1,w_2)
	= w_1(\nabla f_\kappa (W_2)) - \nabla f_\kappa|_x(\nabla_{w_1}W_2) \\
	&= \int_{\R^m} \vf_\kappa(|v|)\{w_1(\nabla(f\circ g_v)(W_2)) - \nabla(f\circ g_v)|_x(\nabla_{w_1}W_2)\} \dv \\
	&= \int_{\R^m} \vf_\kappa(|v|)\nabla^2(f\circ g_v)|_x(w_1,w_2)\dv,
\end{align*}
where we differentiate under the integral sign for the second equality.
\end{proof}

We assume from now on that $f$ is $C^{1,1}$.
Our aim is to get estimates on $f_\kappa$, $\nabla f_\kappa$, and $\nabla^2f_\kappa$.
Clearly, for any $x\in X$ such that $|\nabla f|\le\ell$ on $B(x,r)$,
we have $|f(g_v(x))-f(x)|\le\ell|v|$ for all $0\le|v|\le\kappa<r$ and hence
\begin{align}\label{fkf0}
	|f_\kappa(x) - f(x)| \le \ell\kappa
\end{align}
for all $0<\kappa<r$.

To estimate $\nabla f_\kappa$ at $x\in X$, recall first that $f_\kappa$ does not depend on the choice of frame.
We choose $F$ to be parallel along geodesics through $x$,
that is, we choose a frame at $x$ and extend it via parallel translation along geodesics through $x$.
The computations in the following discussion are based on that choice.
In addition, we assume that $|\nabla f|\le\ell$ and that $|\nabla^2f|\le\beta$ on $B(x,r)$.
Then
\begin{align}\label{df}
	|\nabla f|_x(w) - \nabla f|_{g(x,v)}(w')| \le \beta|v||w| 
\end{align}
for almost all $0<|v|<r$ and all $w\in T_xX$,
where $w'$ denotes the parallel translate of $w$ along the geodesic $c_v$
with initial velocity $v$
and where we note that $g(x,v)=c_v(1)$.
On the other hand, let $c_w$ be the geodesic through $x$ with initial velocity $w$.
Then
\begin{align}\label{varic}
	c = c(s,t) = g(c_w(s),tv) = \exp F(c_w(s),tv)
\end{align}
is a variation of the geodesic $c_v$, where $s$ is the variational parameter.
We write $c_s(t)=c(s,t)$ and consider the Jacobi fields $J_s=\partial c_s/\partial s$
along the geodesics $c_s$.
By the choice of $F$, we have
\begin{align}\label{icj}
	J_s(0) = c_w'(s)
	\hspace{3mm}\text{and}\hspace{3mm}
	J_s'(0)=0.
\end{align}
For any $v\in\R^m$ and $w\in T_xX$,
\begin{align}\label{dgv}
	g_{v*x}(w) = (\partial c/\partial s)(0,1) = J_0(1)
\end{align}
and hence, by the chain rule,
\begin{align}\label{dfgv}
	\nabla(f\circ g_v)|_x(w) = \nabla f|_{g(x,v)}(J_0(1)).
\end{align}
We will use here and below that
\begin{align}\label{estj}
	|J_0(1)-w'| \le (\cosh(b|v|)-1)|w|,
\end{align}
where $w'$ denotes the parallel translate of $w$ along $c_0$;
see \cite[Corollary 6.3.8]{BuserKarcher81}.
From \eqref{df} and \eqref{estj}, we obtain
\begin{align*}
	|\nabla&(f\circ g_v)|_x(w) - \nabla f|_x(w)| \\
	&\le |\nabla(f\circ g_v)|_x(w) - \nabla f|_{g(x,v)}(w')| 
	+ | \nabla f|_{g(x,v)}(w') - \nabla f|_x(w)| \\
	&= |\nabla f|_{g(x,v)}(J_0(1) - w')| + | \nabla f|_{g(x,v)}(w') - \nabla f|_x(w)| \\
	&\le \ell b^2|v|^2 |w| + \beta|v||w|
	\le (\ell b+\beta)|v||w| 
\end{align*}
for almost all $0\le|v|<r$ with $b|v|\le1$ and all $w\in T_xX$,
where we use the rough estimate $\cosh t\le 1+t^2$ for $0\le t\le 1$.
In conclusion,
\begin{align}\label{fkf1}
	|\nabla f_\kappa|_x - \nabla f|_x| \le (\ell b+\beta)\kappa
\end{align}
for all $0<\kappa<r$ with $b\kappa\le1$.

With our special frame through $x$ and geodesic variation $c$ as in \eqref{varic}
and for any $v\in\R^m$ such that $\nabla^2f$ exists at $g(x,v)$,
\begin{align}\label{d2fgv}
\begin{split}
	\nabla^2(f\circ g_v)|_x(w,w)
	&= \frac{\partial}{\partial s}\big|_{s=0} \nabla f|_{g(c_w(s),v)}(J_s(1)) \\
	&= \nabla^2f|_{g(x,v)}(J_0(1),J_0(1)) + \nabla f|_{g(x,v)}(K_0(1)),
\end{split}
\end{align}
where $K_s=\nabla J_s/\partial s$ denotes the covariant derivative
of the Jacobi fields $J_s$ in the variational direction.
Note that $K_0$ satisfies the inhomogeneous Jacobi equation \eqref{jk} along $c_0$
with initial conditions
\begin{align}\label{ick}
	K_0(0) = 0
	\hspace{3mm}\text{and}\hspace{3mm}
	K_0'(0) = R(F(x,v),w)w.
\end{align}
To get an estimate for $K_0(1)$,
set $X=K_0'$ to reduce the differential equation \eqref{jk} for $K_0$ of second order
to the differential equation
\begin{align*}
	K_0' & = X \\
	X' &= -R(K_0,c_0')c_0' + Y
\end{align*}
of first order for the pair $(K_0,X)$, where
\begin{align*}
	Y = 4R(c_0',J_0)J_0' + \nabla R(c_0',c_0',J_0)J_0 - \nabla R(J_0,J_0,c_0')c_0'.
\end{align*}
Since $|c_0'|=|v|$, \cref{rauch2} yields
\begin{align*}
	|J_0|\le\cosh(b|v|)|w| \quad\text{and} \quad |J_0'|\le |v|b\tanh(b|v|)|J_0|
\end{align*}
on $[0,1]$.
Therefore
\begin{align*}
	|(K_0,X)|'
	&\le |(K_0,X)'| \\
	&\le |(K_0,X)| + 4b^2|c_0'||J_0||J_0'| + 2b'|c_0'|^2|J_0|^2 \\
	&\le |(K_0,X)| + (4b^3\tanh(b|v|)+2b')\cosh(b|v|)^2 |v|^2|w|^2 \\
	&\le |(K_0,X)| + 16(b^3 + b')|v|^2|w|^2
\end{align*}
for all $b|v|\le1$, where $b'$ is a bound of $\nabla R$ on $B(x,r)$.
Now Gronwall's inequality implies that, for all $b|v|\le1$,
\begin{align}\label{gw2}
\begin{split}
	|K_0(1)| &\le |(K_0,X)(1)| \\
	&\le (|(K_0,X)(0)| + 16(b^3 + b')|v|^2|w|^2)e \\
	&= (|R(F(x,v),w)w|+ 16(b^3 + b')|v|^2|w|^2)e \\
	&= (b^2|v||w|^2+ 16(b^3 + b')|v|^2|w|^2)e
	\le C|v||w|^2.
\end{split}
\end{align}
With $w'$ as above, we get
\begin{align*}
	\nabla^2&(f\circ g_v)|_x(w,w)
	= \nabla^2f|_{g(x,v)}(J_0(1),J_0(1)) + \nabla f|_{g(x,v)}(K_0(1)) \\
	&= \nabla^2f|_{g(x,v)}(w',w') +  \nabla^2f|_{g(x,v)}(w',J_0(1)-w') \\
	&\hspace{6mm}+ \nabla^2f|_{g(x,v)}(J_0(1)-w',J_0(1)) + \nabla f|_{g(x,v)}(K_0(1)) \\
	&= \nabla^2f|_{g(x,v)}(w',w') \pm \{\beta(\cosh(b|v|)-1)|w|^2 \\
	&\hspace{6mm}+ \beta(\cosh(b|v|)-1)\cosh(b|v|)|w|^2 + C\ell|v||w|^2\}
\end{align*}
for almost all $0\le|v|\le\kappa<r$ with $b|v|\le1$ and all $w\in T_xX$.
Setting $C'=(1+\cosh(1))b$, we conclude that
\begin{align}\label{fkf2}
	\nabla^2(f\circ g_v)|_x(w,w) = \nabla^2f|_{g(x,v)}(w',w')
	\pm (C\ell+C'\beta)\kappa|w|^2
\end{align}
for almost all $0\le|v|\le\kappa<r$ with $b|v|\le1$ and all $w\in T_xX$.

\begin{proof}[Proof of \cref{papa}]
Let $f$ be the distance function to $C=\{f=0\}$.
Then, by \cref{curvest2}, there is a $0<\delta<\ve$ such that  $\{f\le r\}$ is a closed and strictly
convex domain such that the second fundamental form $S_r$ of its boundary $\{f=r\}$
satisfies $0<\alpha-\ve\le S_r\le\beta+\ve$ almost everywhere, for all $0\le r<\delta$. 
Then $|\nabla^2f|\le\beta+\ve$ on $0<f\le\delta$ and therefore, by \eqref{fkf1},
\begin{align}\label{dfkest}
	|\nabla f_\kappa|_x - \nabla f|_x| \le (b+\beta+\ve)\kappa,
\end{align}
for all $0<\kappa<\delta/3$ and $x\in\{\delta/3\le f\le2\delta/3\}$.
We also have
\begin{align*}
	|\nabla^2&(f\circ g_v)|_x(w,w) -  \nabla^2f|_{g(x,v)}(w',w')|
	\le (C+C'(\beta+\ve))\kappa|w|^2
\end{align*}
for almost all $|v|<\delta/3$ with $b|v|\le1$ and $x\in\{\delta/3\le f\le2\delta/3\}$, by \eqref{fkf2}.
Therefore
\begin{align*}
	\alpha - 2\ve
	\le \nabla^2 f_\kappa|_x
	\le \beta + 2\ve
	\quad\text{on $\nabla f_x^\perp$}
\end{align*}
for all sufficiently small $\kappa>0$ and $x\in\{\delta/3\le f\le2\delta/3\}$.
Thus
\begin{align}\label{d2fkest}
	\alpha - 3\ve
	\le \nabla^2 f_\kappa|_x
	\le \beta + 3\ve
	\quad\text{on $\nabla f_\kappa|_x^\perp$}
\end{align}
for all sufficiently small $\kappa>0$ and $x\in\{\delta/3\le f\le2\delta/3\}$.
Therefore $C'=\{f_\kappa\le\delta'\}$ is a closed domain containing $C$ and contained in $U_\ve(C)$,
whose boundary $\partial C'=\{f_\kappa=\delta'\}$ is smooth
with second fundamental form $\alpha-3\ve\le S'\le\beta+3\ve$,
for any sufficiently small $\kappa>0$
and regular value $\delta'$ of $f_\kappa$ sufficiently close to $\delta/2$.
If  $0<3\ve<\alpha$, then $S'$ is positive definite and hence $C'$ strictly convex.
By \cref{invari}, $C'$ is invariant under any isometry of $X$ leaving $C$ invariant.
\end{proof}

\section{The symmetry of second derivatives}
\label{symmetry}

Say that $x\in\R^m$ is a $2$-Lebesgue point of a map $f\colon\R^m\to\R^n$ if,
for any orthonormal $u,v\in\R^m$ tangent to coordinate directions,
\begin{align*}
	\lim_{r\to0}\frac1{r^2}\int_0^r\int_0^r |f(x+su+tv)-f(x)| = 0.
\end{align*}
Together with the Fubini theorem,
the Lebesgue differentiation theorem implies that almost any point of $\R^m$
is a $2$-Lebesgue point of $f$ if $f$ is locally integrable;
compare with \cite[Section 2.1.4]{GiaquintaModica09}.

\begin{lem}\label{lemsym}
For any $f\in C^{1,1}(\R^m,\R^n)$,
$d^2f(x)$ is symmetric at each $2$-Lebesgue point $x$ of the map $d^2f$.
\end{lem}

In view of our needs, we assume that $f$ is $C^{1,1}$,
but somewhat weaker assumptions would also be sufficient.

\begin{proof}[Proof of \cref{lemsym}]
For $u,v\in\R^m$, we have
\begin{align*}
	f(x+ru+rv) - f(x&+ru) - f(x+rv) + f(x) \\
	&= \int_0^r \{df(x+ru+tv)-df(x+tv)\}v\dt \\
	&= \int_0^r\int_0^r d^2f(x+su+tv)(u,v)\ds\dt \\
	&= I_{u,v}(r) + r^2d^2f(x)(u,v),
\end{align*}
where we note, for the penultimate equality, that $df$ is $C^{0,1}$ and where
\begin{align*}
	I_{u,v}(r) = \int_0^r\int_0^r \{ d^2f(x+su+tv)-d^2f(x)\}(u,v)\ds\dt.
\end{align*}
Interchanging the roles of $u$ and $v$, we obtain
\begin{align*}
	f(x+ru+rv) - f(x&+ru) - f(x+rv) + f(x) \\
	&= \int_0^r\int_0^r d^2f(x+su+tv)(v,u)\dt\ds \\
	&= I_{v,u}(r) + r^2d^2f(x)(v,u).
\end{align*}
If $x$ is a $2$-Lebesgue point of $d^2f$ and $u,v$ are orthonormal and tangent to coordinate directions,
then we have
\begin{align*}
	\lim_{r\to0}\frac1{r^2}  I_{u,v}(r)
	= \lim_{r\to0}\frac1{r^2} I_{v,u}(r)
	= 0.
\end{align*}
Therefore, by the above computations,
\begin{align*}
	|d^2f(x)(u,v)-d^2f(x)(v,u)|
	&\le \lim_{r\to0}\frac1{r^2} |I_{u,v}(r)-I_{v,u}(r)| = 0.
\end{align*}
Hence $d^2f(x)$ is symmetric.
\end{proof}

\begin{cor}
For any $k\ge1$ and $f\in C^{k,1}(\R^m,\R^n)$,
$d^{k+1}f(x)$ is symmetric at each $2$-Lebesgue point $x$ of the map $d^{k+1}f$.
\end{cor}

\bibliographystyle{amsplain}
\bibliography{GeFiSpec}

\end{document}